\documentclass[a4paper,11pt,oneside,reqno]{amsart}
\usepackage[margin=1in]{geometry}                
\usepackage{lmodern} 
\usepackage[OT1]{fontenc}
\usepackage{graphicx}
\usepackage{amssymb,amsmath,amsthm}
\usepackage{mathtools}
\usepackage{enumerate}
\usepackage{fix-cm}
\usepackage{hyperref}
\usepackage{color}
\usepackage{dsfont}
\usepackage{bm}
\usepackage{xfrac}
\allowdisplaybreaks

\usepackage{epstopdf}
\newtheorem{theorem}{Theorem}
\newtheorem{remark}[theorem]{Remark}
\newtheorem{proposition}[theorem]{Proposition}
\newtheorem{corollary}[theorem]{Corollary}
\newtheorem{lemma}[theorem]{Lemma}
\newtheorem*{rmk}{Remark}
\newtheorem{defn}{Definition}

\DeclarePairedDelimiter{\pare}{(}{)}

\DeclareMathOperator{\rE}{\mathbb{E}}
\DeclareMathOperator{\rP}{\mathbb{P}}
\DeclareMathOperator{\id}{\mathrm{id}}

\newcommand{\re}{\mathrm{e}}
\newcommand{\rd}{\mathrm{d}}
\newcommand{\ri}{\mathrm{i}}
\begin{document}
\title{Boundary of the Range of a random walk and the F\"olner property}

\author{George Deligiannidis}
\address{Department of Statistics, University of Oxford, Oxford OX1 3LB, UK \newline
and The Alan Turing Institute, 96 Euston Road, London NW1 2DB, UK}
\email{deligian@stats.ox.ac.uk}

\author{S\'ebastien Gou\"ezel}
\address{Laboratoire Jean Leray, CNRS UMR 6629,
Universit\'e de Nantes, 2 rue de la Houssini\`ere, 44322 Nantes, France}
\email{sebastien.gouezel@univ-nantes.fr}

\author{Zemer Kosloff}
\thanks{The research of Z.K. was partially supported by ISF grant No. 1570/17}
\address{Einstein Institute of Mathematics,
Hebrew University of Jerusalem, Edmond J. Safra Campus, Jerusalem 91904,
Israel}
\email{zemer.kosloff@mail.huji.ac.il}

\begin{abstract}
The range process $R_n$ of a random walk is the collection of sites visited
by the random walk up to time $n$. In this work we deal with the question
of whether the range process of a random walk or the range process of a
cocycle over an ergodic transformation is almost surely a F\"olner sequence
and show the following results:
(a) The size of the inner boundary $|\partial R_n|$ of the range of
recurrent aperiodic random walks on $\mathbb{Z}^2$ with finite variance and
aperiodic random walks in $\mathbb{Z}$ in the standard domain of attraction
of the Cauchy distribution, divided by $\frac{n}{\log^2(n)}$, converges to
a constant almost surely.
(b) We establish a formula for the F\"olner asymptotic of transient cocycles
over an ergodic probability preserving transformation and use it to show
that for transient random walk on groups which are not virtually cyclic,
for almost every path, the range is not a F\"olner sequence.
(c) For aperiodic random walks in the domain of attraction of symmetric
$\alpha$- stable distributions with $1<\alpha\leq 2$,  we prove a sharp
polynomial upper bound for the decay at infinity of $|\partial R_n|/|R_n|$.
This last result shows that the range process of these random walks is
almost surely a F\"olner sequence.
\end{abstract}

\maketitle

\section{Introduction}
Let $\mathsf{G}$ be a countable group with identity element $\id_\mathsf{G}$, $\xi_1, \xi_2, \dots$ be i.i.d.\ $\mathsf{G}$-valued random variables and define the random
walk $(S_n)_n$, where $S_0:=\id_\mathsf{G}$ and  $S_n=\xi_1\cdot \xi_2\cdots \xi_n$ for $n\geq 1$. The \textit{range} of the
random walk, denoted
$$R(n):= \{ S_1, \dots, S_{n}\},$$
is the random subset of $\mathsf{G}$ which consists of the sites visited by the random walk up to time $n$. The case where $\mathsf{G}=\mathbb{Z}^d$ will serve as a motivating and recurring example in this paper; as this group is abelian we will denote the random walk in this case by $S_n=\sum_{i=1}^n\xi_i$.

The range  is a natural object to study and understanding its size and shape is of great interest for a variety of models in probability theory such as \textit{random walk in random scenery}; see for example \cite{Aa12,Deligiannidis-Kosloff} where the size of the range is shown to determine the leading term of the asymptotic growth rate of the information arising from the scenery.
The size of the range and its fluctuations have been extensively treated in the
literature starting with the seminal paper \cite{DE51} where the authors obtained strong laws in the case of the simple random walk on $\mathbb{Z}^2$, see also \cite{ET60} and \cite{Flatto}.
A central limit theorem for the range was obtained in \cite{JP72}, whereas the case of random walks with stable increments was treated in \cite{LeR}.

 More recently, the focus has shifted towards more involved objects, still related to the range. For example \cite{Benjamini-Kozma-Rosen-Yehudayoff.} studies the entropy of the range, \cite{asselah2016capacity4,asselah2016capacityd} its \textit{capacity} and finally \cite{AS17,Okada} study the boundary of the range, henceforth denoted by $\partial R_n$, that is the sites in the range with at least one neighbour not visited by the random walk. Apart from its intrinsic interest, the motivation for studying the range and its relatives often stems from their relevance in more intricate models; the capacity
of the range is of interest in \textit{random interlacements} (see \cite{Sznitman}), whereas the range itself is relevant in the study of random walks in random scenery.

The main focus of this paper is on the boundary of the range, and our interest in this particular object is motivated by its relation to the F\"olner property of the range,
which the first and last authors first studied in \cite{Deligiannidis-Kosloff}.
In the case of transient random walks on $\mathbb{Z}^d$ with $d\geq 1$, Okada in \cite{Okada} has proved a law of large numbers result for
the size of the boundary of the range. The starting point in \cite{Okada} is a result of Spitzer which uses the ergodic theorem to show that for all random walks, almost surely
\[
\lim_{n\to\infty}\frac{\left|R_n\right|}{n}=\mathbb{P}\left(S_n\neq 0, \,\text{for all }n\geq 1\right),
\]
where clearly the limit is positive for transient random walks. In Section~\ref{sec: cocycles countable groups} we generalise Okada's result to transient cocycles of ergodic probability measure preserving systems. We first show, see Proposition~\ref{prop: Kesten ext}, that Spitzer's result on the size of the range extends to this setting. Using this we are able to
show that for any transient $\mathsf{G}$-valued cocycle, almost surely
\begin{equation}\label{eq:kaimanovich}
\lim_{n\to\infty}\frac{\left|R_n\triangle R_n g\right|}{\left|R_n\right|}=c(g),
\end{equation}
where $c(g)$ is  given explicitly in terms of return probabilities, see Theorem \ref{thm: Folner for transient} for a precise statement.
The range of the random walk is almost surely a F\"olner sequence if $c(g)=0$ for all $g$.
In the last part of Section~\ref{sec: cocycles countable groups} we focus on random walks and show
that if the Green function of the random walk decays at infinity then the range process is almost surely not a F\"olner sequence.
We show in Appendix~\ref{appendix:Green} that, except for virtually cyclic groups, the Green function of
a transient random walk always tends to $0$ at infinity. Therefore, on any non virtually cyclic group,
the range of any transient random walk is almost surely not a F\"olner sequence. Furthermore, we show that the only transient random walks in $\mathbb{Z}$
satisfying \eqref{eq:kaimanovich} with $c=0$  are the obvious ones, namely those with positive mean
supported on $(-\infty, 1]$ and those with negative mean supported on $[-1, \infty)$, thus answering a question of Kaimanovich.

In the case of the simple random walk on $\mathbb{Z}^2$, the first and last
authors have shown in \cite{Deligiannidis-Kosloff} that the range is almost
surely a F\"olner sequence. As the range is a connected set, for all
$v\in\mathbb{Z}^2$, the random variable $\left|\partial R_n\right|$, the
cardinality of the boundary, can be used to bound $\left|R_n\triangle
\left(R_n+v\right)\right|$ from above. In Section \ref{sec: boundary Z^2} we
show that if $(S_n)_n$ is either a recurrent aperiodic random walk with
finite variance on $\mathbb{Z}^2$, or an aperiodic random walk on
$\mathbb{Z}$ in the standard domain of attraction of the symmetric Cauchy
distribution, there exists $c>0$ such that
\[
\lim_{n\to\infty}\frac{\left|\partial R_n\right|}{\left(n/\log^2(n)\right)}=c,\ \ \text{almost surely}.
\]
We remark that this constant $c$ coincides with that in the asymptotic of
$\rE|\partial R_n|$ and thus it follows from \cite{Okada} that for the simple
random walk on $\mathbb{Z}^2$, $c\in \left(\sfrac{\pi^2}{2},2\pi^2\right)$.
Since the range in these settings is almost surely of  order
$\sfrac{dn}{\log(n)}$ for some $d>0$,  for all $v\in \mathbb{Z}^2$ (or
$\mathbb{Z}$ in the Cauchy case), there exists a $C>0$ such that almost
surely for $n$ large enough,
\[
C^{-1}\left(\frac{1}{\log(n)}\right)\leq \frac{\left|R_n\triangle \left(R_n+v\right)\right|}{\left|R_n\right|}\leq C\left(\frac{1}{\log(n)}\right).
\]
This is a rather precise F\"olner asymptotic for the range process of these random walks.

Finally in Section~\ref{sec: Folner} we consider the range process of
aperiodic, $\mathbb{Z}$-valued random walks  in the domain of attraction of
the symmetric, $\alpha$-stable distribution with $1<\alpha<2$. In these cases
the range, scaled appropriately, converges in distribution to a nontrivial
random variable, namely the Lebesgue measure of the set $W_\alpha
\left([0,1]\right):=\left\{W_\alpha(t):\ 0\leq t\leq 1\right\}$, where
$W_\alpha$ is the symmetric, $\alpha$-stable L\'evy process. For the simple
random walk on $\mathbb{Z}$, $R_n$ is an interval, and thus there is no
scaling to ensure that $\left|\partial R_n\right|/\left|R_n\right|$ converges
almost surely to a positive constant. Theorem~\ref{thm: alpha stable} is a
quantitative upper bound of these quantities which is optimal in the
polynomial term. In particular this shows that for this class of random
walks, the range process is almost surely a F\"olner sequence. This in turn can
be used, for example, to circumvent the use of dyadic partitions in the
calculation of the relative complexity of the scenery in \cite{Aa12} as
Kieffer's Shannon-McMillan-Breiman formula applies directly to the sequence
of sets $(R_n)_n$.

\section{The Range of cocycles in countable groups}\label{sec: cocycles countable groups}
Let $T$ be an invertible, ergodic, measure preserving transformation of a probability space $\left(\Omega,\mathcal{B},\mathbb{P}\right)$ and $(\mathsf{G},\times)$ be a
countable discrete group and denote by $\mathrm{id}_\mathsf{G}$ and $m_\mathsf{G}$ the identity element and the Haar measure of $\mathsf{G}$ respectively.
A $\mathsf{G}$-valued cocycle is a function $F:\mathbb{Z}\times \Omega\to \mathsf{G}$ which satisfies the cocycle identity: For all $ m,n\in\mathbb{N}$ and $\omega\in\Omega$,
\[
F(n+m,\omega)=F(m,\omega) \times F(n,T^m \omega).
\]
Any function $f: \Omega\to \mathsf{G}$  determines a $\mathsf{G}$ valued cocycle $F$ via
\[
F(n,\omega)=\begin{cases}
f(\omega)\times f\circ T(\omega)\times \cdots \times f\circ T^{n-1}(\omega) & n\in \mathbb{N}\\
\id_{\mathsf{G}} & n=0\\
\left(f\circ T^{-1}(\omega)\right)^{-1}\times \left(f\circ T^{-2}(\omega)\right)^{-1}\times \cdots\times  \left(f\circ T^{-n}(\omega)\right)^{-1} & n\in\mathbb{Z}_-
\end{cases}
\]
These cocycles appear in the projection to the second coordinate of the skew product map $T_f:\Omega
\times \mathsf{G}\to \Omega\times \mathsf{G}$ defined by $T_f(\omega,g)=\left(T\omega,gf(\omega)\right)$. Note that $T_f$ preserves the $\sigma$ finite measure
$\mathbb{P}\times m_{\mathsf{G}}$.
The cocycle is recurrent if almost surely $N(\omega):=\# \left\{n\in\mathbb{N};\ F(n,\omega)=\id_{\mathsf{G}} \right\}=\infty$ and transient
if $N(\omega)<\infty$ almost surely.  In \cite{Schmidt}, it is shown that $F$ is recurrent if and only if $T_f$ is conservative, meaning that it satisfies Poincar\'e recurrence.
Denote by
\[
R(n)(\omega)=R_\omega(n)=\left\{F(k,\omega);\ 1\leq k\leq n\right\}
\]
the range (a.k.a. the trace) of the cocycle up to time $N$. When no confusion is possible we will write $R(n)$ for $R_\omega (n)$.
For a finite subset $A\subset \mathsf{G}$, $|A|$ is the cardinality of $A$.
\begin{proposition}\label{prop: Kesten ext}
Let $\left(\Omega,\mathcal{B},\mathbb{P}\right)$ be an ergodic probability preserving transformation, $\mathsf{G}$ a countable group and $f:\Omega\to \mathsf{G}$.
Then for $\mathbb{P}$-almost every $\omega$,
\[
\lim_{n\to\infty}\frac{|R_\omega(n)|}{n}=\mathbb{P}\left(\omega'\in\Omega;\ \forall n\in\mathbb{N},\ F(n,\omega')\neq \id_{\mathsf{G}}\right)
\]
\end{proposition}
\begin{proof}
Define
\[
A_n=\left\{\omega'\in\Omega: \forall k\in [1,n]\cap\mathbb{N},\ F(k,\omega')\neq \id_{\mathsf{G}}   \right\}
\]
and $A=A_\infty$.  Counting each $z\in R(n)$ according to the last time it has been visited in $[1,n]\cap\mathbb{N}$, it follows that
\[
|R_{\omega}(n)|=\sum_{k=1}^{n}\mathbf{1}_{\left[\forall k<j\leq n, F(j,\omega)\neq F(k,\omega) \right]}=
\sum_{k=1}^{n}\mathbf{1}_{\left[\forall 0<j\leq n-k, F(j,T^k \omega)\neq \id_{\mathsf{G}} \right]}
\]
and for all $n\in\mathbb{N}$ and $N<n$
\[
\sum_{k=1}^n\mathbf{1}_A\circ T^k\leq |R(n)|=\sum_{k=1}^{n}\mathbf{1}_{A_{n-k}}\circ T^k\leq N+\sum_{k=1}^{n-N}\mathbf{1}_{A_N}\circ T^k.
\]
By the pointwise ergodic theorem, dividing all sides of the inequality by $n$, one has that for all $N\in\mathbb{N}$, for almost every $\omega\in\Omega$,
\[
\mathbb{P}(A)\leq \varliminf_{n\to\infty} \frac{|R_\omega(n)|}{n}\leq \varlimsup_{n\to\infty}\frac{|R_\omega(n)|}{n}\leq \mathbb{P}\left(A_N\right).
\]
Noting that $A_N\downarrow A$ as $N\to\infty$ and thus $\mathbb{P}\left(A_N\right)\xrightarrow[N\to\infty]{}\mathbb{P}(A)$ the conclusion follows.
\end{proof}
The former proposition was proved in \cite{Spitzer} for random walks in $\mathbb{Z}^d$. That is the case where $T:(\mathbb{Z}^d)^\mathbb{Z}\to (\mathbb{Z}^d)^\mathbb{Z}$
is the Bernoulli shift, $\mathbb{P}$ is a product measure (distribution of an i.i.d. sequence) and $f(\omega)=\omega_0$, where $\omega=(\cdots, \omega_{-1},\omega_0,\omega_1,\cdots)$.
\begin{corollary}\label{cor: Kesten 2}
Let $\left(\Omega,\mathcal{B},\mathbb{P}\right)$ be an ergodic probability preserving transformation, $\mathsf{G}$ a countable group and $f:\Omega\to \mathsf{G}$.	
If $F$ is recurrent then $|R(n)|=o(n)$ almost surely. In the transient case there exists $c>0$ such that $|R(n)|\sim cn$ almost surely.
\end{corollary}
\begin{proof}
It remains to show that if $F$ is transient then $\mathbb{P}\left(\omega\in\Omega:\ \forall n\in\mathbb{N},\ F(n,\omega)\neq \id_{\mathsf{G}}\right)=:c>0$. To see this let $l:\Omega \to \mathbb{N}\cup {0}$
\[
l(\omega):=\sup\left\{n\in\mathbb{N}\cup\{0\};\ F(n,\omega)=\id_{\mathsf{G}} \right\}.
\]
We need to show that $\mathbb{P}(l=0)>0$. Since $F$ is transient there exists $N\in\mathbb{N}\cup\{0\}$ such that $\mathbb{P}(B)>0$ where
$B=\left\{\omega\in\Omega;\ l=N\right\}$. For all $\omega\in B$ and $j\in\mathbb{N}$,
\[
\id_{\mathsf{G}}\neq F(N+j,\omega)=F(N,\omega) \times F\left(j,T^N \omega\right) =F\left(j,T^N \omega\right).
\]
The latter implies that $T^{N}B\subset \left\{\omega\in\Omega; l=0\right\}$ and since $T$ preserves $\mathbb{P}$ we have $\mathbb{P}(l=0)>0$.
\end{proof}
\subsection{The asymptotic size of the boundary of transient cocycles}
In \cite{Benjamini-Kozma-Rosen-Yehudayoff.}, \cite{Okada}, the asymptotic size of the boundary of a random walk in $\mathbb{Z}^d$, was studied. More specifically let
\[
\partial R(n)=\left\{x\in R(n): \exists y\in\mathbb{Z}^d\setminus R(n),\ |x-y|=1\right\}
\]
and for $v\in\mathbb{Z}^d$,
\[
\partial_v R(n)= R(n)\setminus \left(R(n)+v\right).
\]
One of the main results of \cite{Okada} is that if $X_1,X_2,\cdots,$ is a sequence of i.i.d.\ symmetric random variables in $\mathbb{Z}^d$, $d\geq 3$, then almost surely, writing $S_k:=\sum_{j=1}^kX_j$ and  $R(n)=\left\{S_k;\ 1\leq k\leq n\right\}$,
\[
\lim_{n\to\infty}\frac{\left|\partial{R(n)}\right|}{n}=
\mathbb{P}\Big(\Big\{\forall k\in \mathbb{N}, S_k\neq 0\Big\}\cap\Big\{ \exists y\in
\left\{\pm e_i;\ 1\leq i\leq d\right\}, \forall k\in\mathbb{Z},\ S_k\neq y\Big\} \Big),
\]
where we write $e_1, \dots, e_d$ for the usual basis vectors of $\mathbb{Z}^d$.
The following proposition is a cocycle version of this result whose proof is almost identical to \cite{Okada}.
\begin{proposition}
\label{prop:range_of_Birkhoff} Let
$\left(\Omega,\mathcal{B},\mathbb{P}\right)$ be an ergodic probability
preserving transformation, $\mathsf{G}$ a countable group and $f:\Omega\to
\mathsf{G}$. Then for all $g\in \mathsf{G}$,  for $\mathbb{P}$-almost every
$\omega\in\Omega$,
\[
	\lim_{n\to\infty}\frac{\left|R_\omega(n)\setminus R_\omega(n)g\right|}{n}=\mathbb{P}\left(\omega\in\Omega;\ \forall n\in\mathbb{N},\ F(n,\omega)\neq \id_{\mathsf{G}},
	\forall n\in\mathbb{Z},\ F(n,\omega)\neq g^{-1}\right)
\]
\end{proposition}
\begin{proof}
Let $g\in \mathsf{G}$ and $\omega\in \Omega$. By definition, $z\in R_\omega(n)\setminus R_\omega(n) g$ if and only if there exists
$1\leq k\leq n$ such that  for all $1\leq j\leq n$, $$F(j,\omega) g\neq z=F(k,\omega).$$
By considering the maximal $k\leq n$ for which $F(k,\omega)=z$, it holds that $\left|R(n)\setminus R(n) g\right|=\sum_{k=1}^n \mathbf{1}_{B_n(k)}$ where
\[
B_n(k)=\left\{\omega\in\Omega; \forall j\in[1,n]\cap\mathbb{N}, F(j,\omega)\neq F(k,\omega)g^{-1},\ \forall k<j\leq n,\ F(j,\omega)\neq F(k,\omega) \right\}
\]
As for all $k,j\in \mathbb{Z}$ , $F(k,\omega)^{-1} F(j,\omega)=F(j-k,T^k\omega)$, for all $k\in\left\{1,\ldots,n\right\}$,
\begin{align*}
B_n(k):=T^{-k}\Bigl\{\omega\in \Omega:\ & \forall j\in[-k,n-k]\cap \mathbb{Z},\ F(j,\omega)\neq g^{-1},
\\ & \forall j\in[1,n-k]\cap \mathbb{N},\ F(j,\omega)\neq \id_{\mathsf{G}}\Bigr\}
\end{align*}
This implies that for all $N<n/2$,
\[
\sum_{k=1}^n 1_C\circ T^k \leq \left|R(n)\setminus R(n)g\right|\leq 2N+\sum_{k=N+1}^{n-N}1_{C_N}\circ T^k,
\]
where
\[
C_N:=\left\{\omega\in \Omega;\ \forall j\in[-N,N]\cap \mathbb{Z},\ F(j,\omega)\neq g^{-1},\ \forall j\in[1,N]\cap \mathbb{N},\ F(j,\omega)\neq \id_{\mathsf{G}}\right\}
\]
and $C=C_\infty=\cap_N C_N$. The conclusion follows from the ergodic theorem as in the proof of Proposition \ref{prop: Kesten ext}.
\end{proof}
\begin{corollary}\label{cor: boundary transient}
Let $\left(\Omega,\mathcal{B},\mathbb{P}\right)$ be an ergodic probability preserving transformation, $\mathsf{G}$ a countable group and $f:\Omega\to \mathsf{G}$.	
If $F$ is transient then for all $g\in \mathsf{G}$,
\[
\frac{\left|R(n)\setminus R(n)g\right|}{|R(n)|}\xrightarrow[n\to\infty]{a.s.}\mathbb{P}\left(\left.\forall k\in\mathbb{Z},\ F(k,\cdot)\neq g^{-1}\right| \forall k\in\mathbb{N},\ F(k,\cdot)\neq \id_{\mathsf{G}}\right).
\]
\end{corollary}
\begin{proof}
Follows from a combination of the last Proposition, Proposition~\ref{prop: Kesten ext} and Corollary~\ref{cor: Kesten 2}.
\end{proof}
A sequence $\left\{F_n\right\}_{n\in\mathbb{N}}$ of subsets of $\mathsf{G}$ is a \textit{right F\"olner sequence} in $\mathsf{G}$ if for all $n\in\mathbb{N}$, $F_n$ is a finite set and for all $g\in \mathsf{G}$,
\[
\frac{\left|F(n)\triangle F(n)g\right|}{|F(n)|}\xrightarrow[n\to\infty]{}0.
\]
The existence of F\"olner sequences is equivalent to amenability of the group (existence of a right invariant mean on $\mathsf{G}$).
In \cite{Deligiannidis-Kosloff}, it was shown that the range of the symmetric random walk in $\mathbb{Z}^2$ is almost surely a F\"olner sequence.
Kaimanovich, in a private communication, suggested to check which cocycles have almost surely F\"olner ranges. The following is a partial advance on this question for transient cocycles.

\begin{theorem}\label{thm: Folner for transient}
Let $\left(\Omega,\mathcal{B},\mathbb{P}\right)$ be an ergodic probability preserving transformation, $\mathsf{G}$ a countable group and $f:\Omega\to \mathsf{G}$.	If $F$ is transient then for all $g\in \mathsf{G}$, almost surely,
\[
\lim_{n\to\infty}\frac{\left|R(n)\triangle R(n)g\right|}{|R(n)|}=\Phi(g)+\Phi\left(g^{-1}\right)
\]
where $\Phi(g):=\mathbb{P}\left(\left.\forall n\in\mathbb{Z},\ F(n,\cdot)\neq g\right| \forall k\in\mathbb{N},\ F(k,\cdot)\neq \id_{\mathsf{G}}\right)$.
\end{theorem}
\begin{proof}
It is easy to see, by multiplying by $g^{-1}$, that for all $n\in\mathbb{N}$,
\[
\left|R(n)g\setminus R(n)\right|=\left|R(n)\setminus R(n)g^{-1}\right|.
\]
Since $\left(R(n)\triangle R(n)g\right)=\left(R(n)\setminus R(n)g\right)\uplus\left(R(n)g\setminus R(n)\right)$, the result follows from Corollary \ref{cor: boundary transient}.
\end{proof}
\subsection{F\"olner property and transient random walks}
\label{subsec:transient_RW}
We now turn to an application of Theorem \ref{thm: Folner for transient} in the context of random walks. \\
Let $\mathsf{G}$ be a countable group. Given $p$ a probability measure on $\mathsf{G}$ and $\xi_1,\xi_2,...$ an i.i.d.\ sequence with marginals distributed like $p$,
let $S_n=\xi_1\xi_2\cdots \xi_n$ be the corresponding random walk and $R(n):=\left\{S_1,...,S_n\right\}$ be its range process.
The ergodic theoretic model of the random walk is the following skew product transformation.
Let $\Omega:=\mathsf{G}^\mathbb{Z}$, $\mathbb{P}=p^{\otimes \mathsf{G}}$ the product measure on $\Omega$
with marginals distributed like $p$ and $T:\Omega\to\Omega$ the full shift defined by $(T\omega)_n=\omega_{n+1}$.
Writing $m_{\mathsf{G}}$ for the Haar measure of $\mathsf{G}$, and $f:\Omega\to \mathsf{G}$, $f(\omega):=\omega(0)$, the skew product transformation
$T_f:\Omega\times \mathsf{G}\to \Omega\times \mathsf{G}$ satisfies
\[
\pi_{\mathsf{G}}(T_f^n)\overset{d}{=}S_n,
\] here $\pi_{\mathsf{G}}(\omega,h)=h$ is the projection to the $\mathsf{G}$ coordinate. The advantage of working with the skew product is that the cocycle identity indicates what is the
relevant random walk in inverse time.  In this case, write for $n\in\mathbb{Z}$,
\[
S_n^{(-)}=S_n^{(-)}(\omega):=F(-n,\omega)=\omega(-1)^{-1}\omega(-2) ^{-1}\cdots \omega(-n)^{-1}.
\] Corollary \ref{cor: boundary transient} gives the following extension of Okada's result.
\begin{corollary}\label{cor: Okada :(}
	Let $\mathsf{G}$ be a discrete countable  group and $p$ a probability measure on $\mathsf{G}$ and $\xi_1,\xi_2,...$ an i.i.d.\ sequence with
	marginal $p$. Then almost everywhere
\begin{align*}
	&\lim_{n\to\infty}\frac{ \left|R(n)\triangle \left(R(n)\cdot g\right)\right|}{|R(n)|}\\
	&= \mathbb{P}\left(\forall n\in\mathbb{N},\ S_n^{(-)}\neq g \right)\mathbb{P}\left(\left. \forall n\in\mathbb{N}, S_n\neq g\right| \forall n\in\mathbb{N}, \ S_n\neq \id_\mathsf{G} \right)\\
	&\ \ \ + \mathbb{P}\left(\forall n\in\mathbb{N},\ S_n^{(-)}\neq g^{-1}\right)\mathbb{P}\left(\left. \forall n\in\mathbb{N}, S_n\neq g^{-1}\right| \forall n\in\mathbb{N}, \ S_n\neq \id_\mathsf{G} \right).
\end{align*}
\end{corollary}
\begin{proof}
This is a direct consequence of Corollary \ref{cor: boundary transient} and the fact that $\{S_n\}_{n=1}^\infty$ and $\{S_n^{(-)}\}_{n=1}^\infty$ are independent.
\end{proof}
\begin{remark}
	In the case where $\mathsf{G}$ is Abelian, $S_n^{(-)}\overset{d}{=}\left(S_n\right)^{-1}$ and thus for all $g\in \mathsf{G}$,
	\begin{align*}
	\mathbb{P}\left(\forall n\in\mathbb{N},\ S_n\neq g^{-1} \right)&= \mathbb{P}\left(\forall n\in\mathbb{N},\ S_n^{(-)}\neq g \right).
	\end{align*}
	The statement of Corollary \ref{cor: Okada :(} can be simplified in this case. Note that for a general group
	$\smash{(S_n^{(-)})^{-1}}$ is a (multiplication from the) left random walk and $S_n$ is a (multiplication from the) right random walk and their distribution
	as processes may no longer coincide.
\end{remark}
Let us start with the case of $S_n$ a transient random walks on $\mathbb{Z}$.
In this setting, if $p$ has finite first moment, by Atkinson's Theorem
\cite{Atkinson} and the law of large numbers, either
$\lim_{n\to\infty}S_n=\infty$ almost surely or $\lim_{n\to\infty}S_n=-\infty$
almost surely. If $\lim_{n\to\infty} S_n=\infty$ almost surely and
\[
\mathbb{P}\left(X_1>1\right)=0,
\]
then for almost all $\omega\in\Omega$, $\left\{R(n)\right\}_{n=1}^\infty$ is
an eventually monotone sequence of growing intervals with $|R(n)|\to\infty$
as $n\to\infty$. Consequently the range process is almost surely a F\"olner
sequence. A similar statement holds for walks with $S_n\to-\infty$ as
$n\to\infty$ and $\mathbb{P}\left(X_1<-1\right)=0$.  We now show this is the
only possibility for $\mathbb{Z}$-valued random walks with almost surely
F\"olner ranges.
\begin{theorem}
Let $S_n=\sum_{k=0}^{n-1}X_k$ be a $\mathbb{Z}$-valued random walk such that
$\lim_{n\to\infty} S_n=\infty$ almost surely and
$\mathbb{P}\left(X_1>1\right)>0$, then almost surely
$\left\{R(n)\right\}_{n=1}^\infty$ is not a F\"olner sequence.
\end{theorem}
\begin{proof}
As the random walk is transient and $S_n\to \infty$ almost surely,\footnote{See the proof of Corollary \ref{cor: Kesten 2} where it shown that $\mathbb{P}(l=0)>0$.}
\[
\mathbb{P}\left(\forall n\in\mathbb{N},\ S_n>0 \right)>0.
\]
In addition there exists $\mathbb{Z}\ni j>1$ such that $\mathbb{P}\left(X_1=j\right)>0$. Therefore,
\begin{align*}
\mathbb{P}\left(\forall n\in\mathbb{N},\ S_n>1\right)&\geq \mathbb{P}\left(X_0 =j \text{ and } \forall 2\leq n\in\mathbb{N}, S_n-X_0>0\right)\\
&=\mathbb{P}\left(X_0=j\right) \mathbb{P}\left(\forall n\in\mathbb{N},\  S_n>0\right),\ \ \text{by the Markov property of}\ S_n.
\end{align*}
It follows that $\mathbb{P}\left(\forall n\in\mathbb{N},\ S_n>1\right)>0$ and
\[
\mathbb{P}\left(\left. \forall n\in\mathbb{N}, S_n\neq 1 \right| \forall n\in\mathbb{N}, \ S_n\neq 0 \right)\geq\frac{\mathbb{P}\left(X_1=j\right) \mathbb{P}\left(\forall n\in\mathbb{N},\  S_n>0\right)}{\mathbb{P}\left(\forall n\in\mathbb{N},\  S_n\neq 0\right)}>0.
\]
As the distributions of $\left\{S_n^{(-)}\right\}_{n=1}^\infty$ and $\left\{-S_n\right\}_{n=1}^\infty$ are the same
\begin{align*}
\mathbb{P}\left(\forall n\in\mathbb{N},\ S_n^{(-)}\neq 1 \right)&=\mathbb{P}\left(\forall n\in\mathbb{N},\ S_n\neq -1 \right)\geq \mathbb{P}\left(\forall n\in\mathbb{N},\ S_n>0\right)>0.
\end{align*}
We have shown that
\[
\mathbb{P}\left(\forall n\in\mathbb{N},\ S_n^{(-)}\neq 1\right)\mathbb{P}\left(\left. \forall n\in\mathbb{N}, S_n\neq  1 \right| \forall n\in\mathbb{N}, \ S_n\neq 0 \right)>0.
\]	
By Corollary \ref{cor: Okada :(}, $\lim_{n\to\infty}\left|R(n)\triangle \left(R(n)+1\right)\right|/ \left|R(n)\right|>0$ almost surely and  $R(n)$ is almost surely not a F\"olner sequence.
\end{proof}
For other types of groups, it is most often the case that for any transient
random walk on $\mathsf{G}$, its Green function decays at infinity, see
Appendix~\ref{appendix:Green}. We next explain what we mean by this and show
that in this case the range of a random walk is almost surely not a F\"olner
sequence. This is applied in Corollary~\ref{cor:not_virtually_cyclic}  to
show that for all admissible transient random walks on non-virtually cyclic
groups, the range is almost surely not a F\"olner sequence.

The Green function of a transient random walk on $\mathsf{G}$ is the function $G:\ \mathsf{G}\to [0,\infty]$ defined by
\begin{equation}
\label{eq:define_Green}
G(g):=\mathbb{E}\left(L(g)\right)=\sum_{n=0}^\infty \mathbb{P}\left(S_n=g\right)
\end{equation}
where $L(g):=\left|\{n\in\mathbb{N} \cup \{0\},\ S_n=g\right\}|:\Omega\to\mathbb{N}\cup\{0\}$. For $g\in G$ denote by
\[
q(g):=\mathbb{P}\left(\forall n\in \mathbb{N},\ S_n\neq g\right),
\]
and write $q:=q\left(\id_\mathsf{G}\right)$.
\begin{lemma}\label{lem: Green}
If $S_n$ is a $\mathsf{G}$-valued transient random walk then for all $g\in \mathsf{G}$,
\[
G(g)=\frac{1-q(g)}{q}
\]
\end{lemma}
\begin{proof}
Let $g\in \mathsf{G}$. By the Markov property for the random walk,
\begin{align*}
\mathbb{P}(L(g) = 0) &= q(g)\\
\mathbb{P}(L(g) = j) &= (1-q(g))(1-q)^{j-1}q,\quad j\geq 1.
\end{align*}

and thus
\begin{align*}
\rE (L(g)) &= q(g)\cdot 0 + (1-q(g))  \sum_{j=1}^\infty j (1-q)^{j-1} q \\
&=	(1-q(g)) \rE L(0) = \frac{1-q(g)}{q}.
\qedhere
\end{align*}
\end{proof}
Given a random walk $S_n$ on $\mathsf{G}$, $G^{(-)}$ denotes the Green function of $S_n^{(-)}$.
\begin{theorem}\label{thm:transientgroup}
If $S_n$ is a $\mathsf{G}$ valued transient random walk and there exists $\left\{g_n\right\}_{n=1}^\infty$ such that $\lim_{n\to\infty}\max\left(G\left(g_n\right),G^{(-)}\left(g_n\right)\right)=0$, then $\left\{R(n)\right\}_{n=1}^\infty$ is almost surely not a F\"olner sequence.
\end{theorem}
\begin{proof}
It follows from Lemma \ref{lem: Green} that $G\left(g_n\right)\xrightarrow[n\to\infty]{}0$ implies that $q\left(g_n\right)\xrightarrow[n\to\infty]{}1$.
Letting
$$A_n:= \left\{\forall k\in\mathbb{N},\ S_k\neq g_n\right\}, \qquad B:= \left\{\forall k\in\mathbb{N},\ S_k\neq \id_\mathsf{G} \right\},$$
we have that $\mathbb{P}(A_n)=q(g_n)\to 1$ and $\mathbb{P}(B)=q$, whence it easily follows that
\[
\lim_{n\to\infty}\mathbb{P}\left(\left.\forall k\in\mathbb{N},\ S_k\neq g_n\right|\forall k\in\mathbb{N},\ S_k\neq \id_\mathsf{G}\right)=
\lim_{n\to\infty}\frac{\mathbb{P}\left(A_n\cap B\right)}{\mathbb{P}(B)}=1.
\]
A similar reasoning shows that $G^{(-)}\left(g_n\right)\xrightarrow[n\to\infty]{}0$ implies that
\[
\lim_{n\to\infty}\mathbb{P}\left(\forall k\in\mathbb{N},\ S_k^{(-)}\neq g_n\right)=1.
\]
By this
\[
\lim_{n\to\infty}\left[\mathbb{P}\left(\forall k\in\mathbb{N},\ S_k^{(-)}\notin g_n\right)\mathbb{P}\left(\left.\forall k\in\mathbb{N},\ S_k\ne	g_n\right|\forall k\in\mathbb{N},\ S_k\neq \id_\mathsf{G}\right)\right]=1
\]
and an application of Corollary~\ref{cor: Okada :(} shows that for all large $n\in\mathbb{N}$, almost surely
\[
\lim_{k\to\infty}\frac{\left|R(k)\triangle\left(R(k) \cdot g_n\right)\right|}{\left|R(k)\right|}>\frac{1}{2}.
\]
We conclude that almost surely the range is not a F\"olner sequence.
\end{proof}

Let us say that a random walk driven by a measure $p$ on the group
$\mathsf{G}$ is \emph{admissible} (or \emph{irreducible}) if the walk
starting from the identity can reach any point in the group, i.e., the
semigroup generated by the support of $p$ is the whole group. This is a
natural irreducibility assumption. We prove in Appendix~\ref{appendix:Green}
that, except for virtually cyclic groups, the Green function of a transient
admissible probability measure tends to $0$ at infinity.

\begin{corollary}
\label{cor:not_virtually_cyclic}
Let $S_n$ be a transient admissible random walk on a group $\mathsf{G}$ which is not virtually cyclic.
Then almost surely $\left\{R(n)\right\}_{n=1}^\infty$ is not a F\"olner sequence.
\end{corollary}
In particular, for all $d\geq 3$,   the range of the simple random walk on $\mathbb{Z}^d$ is almost surely not a F\"olner sequence.
\begin{proof}
This follows directly from the previous Theorem~\ref{thm:transientgroup} and
Theorem~\ref{thm:Green_vanishes} showing that the Green functions $G$ and
$G^{(-)}$ tend to $0$ at infinity.
\end{proof}

\section{Precise F\"olner asymptotics for recurrent planar random walks and Cauchy random walks in
\texorpdfstring{$\mathbb{Z}$}{Z}}\label{sec: boundary Z^2} Let
$\xi_0,\xi_1,\cdots$ be i.i.d.\ $\mathbb{Z}^d$-valued random variables
satisfying one of the following two assumptions
\begin{itemize}
	\item[(A1)\label{ass1}] $d=2$ and there exists a nonsingular covariance matrix $\Sigma\in M_{2\times2}\left(\mathbb{R}\right)$ such that for all $t\in [-\pi,\pi]^2$,
	\[
	\phi(t)=\mathbb{E}\left(\exp(i\left<t,\xi\right>)\right)=1-\left<\Sigma t,t\right>+o\left(|t|^2\right).
	\]
	\item[(A2)\label{ass2}] $d=1$ and there exists $\gamma>0$ such that for $t\in[-\pi,\pi]$,
	\[
	\phi(t)=\mathbb{E}\left(\exp(it\xi)\right)=1-\gamma|t|+o(|t|).
	\]
\end{itemize}
We will always assume that the random walk is \textit{aperiodic}, namely that there is no proper subgroup of $\mathbb{Z}^d$ containing the support of $\xi_1$.
By \cite[Theorem~7.1]{Spitzer} this is equivalent to requiring that $\phi(t)=1$ for $t\in [-\pi,\pi]^d$ if and only if $t=0$.

We will often need the stronger notion of \textit{strong aperiodicity}, which
is equivalent to  $|\phi(t)|< 1$ for all $t\in (-\pi,\pi)^d\setminus\{0\}$.
For aperiodic random walks, it amounts to the condition that the greatest
common divisor of return times to the origin is $1$. For instance, the simple
random walk is aperiodic, but not strongly aperiodic.

For $v\in \mathbb{Z}^d$, we define the \textit{$v$-boundary} of the range as
$$\partial_v R(n):= R(n)\setminus \big(R(n)+v\big).$$
For $d=1,2$, write $\mathds{E}_d:=\{\pm e_i: 1\leq d\}$, where $e_i$ are the
usual generators of $\mathbb{Z}^d$. The boundary is defined as
\[
\partial R(n)=\left\{x\in R(n): \exists y\in\mathbb{Z}^d\setminus R(n),\ |x-y|=1\right\},
\]
which can also be written in the form
\begin{equation}
\label{eq:partialR_union}
\partial R(n) =\bigcup_{v\in \mathds{E}_d} R(n) \setminus (R(n)+v) =\bigcup_{v\in \mathds{E}_d} \partial_v R(n).
\end{equation}
%

The following is the main result of this section.
\begin{theorem}\label{thm: boundary range}
Let $\xi_0,\xi_1,...,\xi_n,..$ be  i.i.d.\ $\mathbb{Z}^d$-valued aperiodic
random variables satisfying either Assumption~(A1) or Assumption~(A2). Then
for every $v\in\mathbb{Z}^d$, there exist constants $C_v, C>0$ such that
almost surely,
\[
	\lim_{n\to\infty} \frac{\log^2(n)}{n}\left|\partial_v R(n)\right|=C_v, \qquad
	\lim_{n\to\infty} \frac{\log^2(n)}{n}\left|\partial R(n)\right|=C.
\]
In addition $\mathbb{E}(\left|\partial_v R(n)\right|)\sim
\frac{C_vn}{\log^2(n)}$ and $\mathbb{E}(\left|\partial R(n)\right|)\sim
\frac{C n}{\log^2(n)}$ as $n\to\infty$.
\end{theorem}

The statements in Theorem~\ref{thm: boundary range} for $\partial_v R(n)$ and
$\partial R(n)$ are proved in a parallel way, the proofs being slightly more
involved for $\partial R$ instead of $\partial_v R$ because of the union
in~\eqref{eq:partialR_union}. We will therefore concentrate on the case of
$\partial R$.

In \cite{Benjamini-Kozma-Rosen-Yehudayoff.} it was shown that in the case of
the symmetric random walk on $\mathbb{Z}^2$, $\rE|\partial R(n)|$ is
proportional to the entropy of the range (at time $n$) and it is of order
constant times $n/\log^2(n)$. Okada in \cite{Okada} has shown that
$2^{-1}\pi^2\leq C\leq 2\pi^2$. The previous Theorem is a law of large
numbers type result for a more general class of random walks which includes
random walks in the domain of attraction of the (symmetric) Cauchy
distribution. It also gives a more precise estimate on the almost sure F\"olner
property of the range of such random walk, see \cite{Deligiannidis-Kosloff}
for some consequences of the F\"olner property of the range in the model of
random walks in random sceneries. The proof goes by first establishing an
upper bound for the variance of $\left|\partial R(n)\right|$ which gives the
convergence in probability. After that we use a method from \cite{Flatto}
which improves the asymptotics of the rate of convergence in probability,
thus enabling us to use the Borel Cantelli lemma for showing convergence
almost surely.

\subsection{Auxiliary results}
We now state and prove a number of auxiliary results that will be used in the proof of Theorem~\ref{thm: boundary range}, that we could not find in the literature.
\begin{proposition}\label{prop: return tail}
	Under the assumptions of Theorem~\ref{thm: boundary range}, for all
	 $j\in\mathbb{Z}^d\setminus\{0\}$ there exist $c_j,d_j>0$ such that
	\[
	\mathbb{P}\left(\forall 1\leq k\leq n,\ S_k\notin \{0,j\}\right)=\frac{c_j}{\log(n)}+o\left(\log(n)^{-1}\right)\ \text{as}\ n\to\infty.
	\]
	and
	\[
	\mathbb{P}\left(\forall 1\leq k\leq n,\ S_k\neq j\right)=\frac{d_j}{\log(n)}+o\left(\log(n)^{-1}\right)\ \text{as}\ n\to\infty.
	\]
\end{proposition}
\begin{rmk}
	The previous proposition in the case of symmetric random walks was (essentially) treated in \cite{Okada}.
\end{rmk}
\begin{proof}
The proof is a simple application of \cite[Theorem~4a]{KS63}. This theorem states that for any aperiodic random walk in $\mathbb{Z}^d$, $W\subset\mathbb{Z}^d$ a finite subset,
and any $x\in \mathbb{Z}^d$ we have
\begin{equation}\label{eq:KSlimit}
\lim_{n\to\infty} \frac{ \rP^x\left( S_k \notin W, k=1, \dots, n \right)}
				{\rP^0\left( S_k \neq 0, k=1, \dots, n \right)} = \tilde{g}_W(x),
\end{equation}
where for the random walks of interest to us, by \cite[Eq.(1.16)]{KS63}, and $x,y\in \mathbb{Z}^d$
\begin{align}
\tilde{g}_W (x,y)
&:= \mathbf{1}_{[x=y]}+ \sum_{n=1}^\infty \rP^x \left( S_n=y, S_k \notin W, 1\leq k \leq n-1 \right),\notag\\
\tilde{g}_W (x)
&= \lim_{|y|\to \infty} \tilde{g}_W (x,y).\label{eq:gw}
\end{align}
Under assumptions (A1), (A2), it follows from the local limit theorem that
\begin{equation}
\sum_{j=0}^n \mathbb{P}(S_j=0) \sim c \log n,\label{eq:partialgreen}.
\end{equation}
Using this asymptotic, a standard argument, see for example \cite{DE51},
shows that there exists a constant $\gamma_d>0$, depending on the random
walk, such that
\begin{equation}\label{eq:noreturn}
\rP^0 \left\{ S_k\neq 0, \, k=1, \dots, n\right\} \sim \frac{\gamma_d}{\log n}.
\end{equation}
When $W=\{0,j\}$, from \cite[Eq.(5.17) and (5.18)]{KS63} we have that
$$\tilde{g}_{\{0,j\}}(j) = \frac{a(j)}{a(j)+a(-j)}, \quad \tilde{g}_{\{0,j\}}(0) = \frac{a(-j)}{a(j)+a(-j)},$$
where
$$a(j) = \sum_{n=0}^\infty \left[\rP^0(S_n=0) - \rP^j(S_n=0)\right]
=\frac{1}{(2\pi)^d} \int_{[-\pi,\pi]^d} \frac{1-\exp(\ri j t)}{1-\phi(t)}\rd t.$$
Under our assumptions, we have $a(x)>0$ for all $x\neq 0$, by \cite[Propositions~11.7 and 30.2]{Spitzer}, whence $c_j,d_j>0$.
The result follows from this and \eqref{eq:KSlimit}.
\end{proof}

\begin{lemma}
\label{lem:Q_control} Under the assumptions of Theorem~\ref{thm: boundary
range}, for any nonempty finite subset $O$ of $\mathbb{Z}^d$ and any $x,y\in
\mathbb{Z}^d$, there exists $C=C(O,x,y)$ such that for all $n \in
\mathbb{N}$,
\begin{equation*}
  \rP^x \left[ S_n=y, S_j \notin O, \, 1\leq j \leq n-1\right]
  \leq \frac{C}{n \log^2(n)}.
\end{equation*}
\end{lemma}
\begin{proof}
Let
$$Q_O^n (x,y):= \rP^x \left[ S_n=y, S_j \notin O, \, 1\leq j \leq n-1\right].$$
Assume first that (A1) holds, i.e., $d=1$, and that the walk is strongly
aperiodic. By \cite[Theorems~8 and 9]{Kes} we have
\begin{align*}
 \lim_{n\to \infty}\frac{Q_{O}^n (x,y)}{\pi \gamma} n \log(n)^2 =1 ,
\end{align*}
which implies the result.
	
Under (A2), still with strong aperiodicity, the estimate
\begin{align*}
  \sum_{u,v\in O}Q_{O}^n (u,v) \sim \frac{2\pi \mathrm{det}(\Sigma)^{1/2}}{n \log(n)^2},
\end{align*}
follows from \cite[Theorem~9]{Kes} and \cite[Theorem~4.1]{JP72}. Then
\begin{align*}
Q_{O}^n (x,y) \sim \frac{2\pi \mathrm{det}(\Sigma)^{1/2}}{n \log(n)^2}\tilde{g}_O(x)
\tilde{g}_{-O}(-y),
\end{align*}
follows from the above and \cite[Theorem~6a]{Kes}. Indeed, although
\cite[Theorem~9]{Kes} is stated only for one-dimensional walks in the domain
of attraction of a  symmetric stable law, as explained in the proof the
result remains true for any recurrent random walk satisfying
\cite[Equation~(11.1)]{Kes}. This has been established under (A2) for
strongly aperiodic random walks in \cite[Theorem~4.1]{JP72}.

Under (A1) or (A2), but with strong aperiodicity, we have proved the result
of the lemma. We claim that it still holds if one weakens strong aperiodicity
to aperiodicity, but a little extra work is needed. Letting $\{S_j\}_{j\geq
0}$ denote the original random walk, we denote with $\{\tilde{S}_j\}_{j\geq
0}$ a \textit{lazy} version of it. In particular, if $S_j=\xi_1+\cdots
+\xi_j$, we let $\tilde{S}_j=\tilde{\xi}_1+\cdots +\tilde{\xi}_j$, where
$\tilde{\xi}_j=B_j \xi_j$ where $\rP\{B_j=1\}=\rho \in (0,1)$ and
$\rP\{B_j=0\} = 1-\rho$. It can be easily checked that $\{\tilde{S}_j\}_{j
\geq 0}$ is then strongly aperiodic and therefore that
\begin{align*}
  \tilde{Q}_{O}^n (x,y)
  &:=\rP^x \left[ \tilde{S}_n=y, \tilde{S}_j \notin O, \, 1\leq j \leq n-1\right] \leq \frac{C}{n \log(n)^2}.
\end{align*}
In addition let
$$T_0:=0, \qquad T_n:=\inf\{n> T_{n-1}: \tilde{\xi}_n \neq 0\}, \, n\geq 1.$$
It is then clear that for all $j\geq 1$, $T_{j}-T_{j-1}$ are i.i.d.\
geometrically distributed on the positive integers. With this notation, we
can embed a path of $\{S_j\}_{j\geq 0}$ into a path of
$\{\tilde{S}_j\}_{j\geq 0}$ by letting $S_j = \tilde{S}_{T_j}$ for $j\geq 0$.
Up to a time-change, the paths of $\{S_j\}$ and $\{\tilde{S}_j\}$ coincide
and thus
\begin{align*}
Q_O^n (x,y)
&= \rP^x \left[ S_n=y, S_j \notin O, \, 1\leq j \leq n-1\right]\\
&=\rP^x \left[ \tilde{S}_{T_n}=y, S_j \notin O, \, 1\leq j \leq T_n-1\right]\\
&=\sum_{k=n}^{\infty}\rP^x \left[ \tilde{S}_{k}=y, S_j \notin O, \, 1\leq j \leq k-1\right]\rP(T_n=k)\\
&\leq C\sum_{k=n}^{\infty}\frac{\rP(T_n=k)}{k \log(k)^2}\leq \frac{C}{n \log(n)^2}\sum_{k=n}^{\infty}\rP(T_n=k) \leq \frac{C}{n \log(n)^2}.
\qedhere
\end{align*}
\end{proof}


As $\mathbb{Z}^d$ is an Abelian group, in this case we consider instead a bi-infinite i.i.d.\ sequence $\{X_i\}_{i\in \mathbb{Z}}$ and write for $n\in\mathbb{N}$,
\[
S_n^{(-)}=-\sum_{k=-n}^{-1}X_k.
\]
In what follows we will make use of the fact that for $k<n$ and $v\in\mathbb{Z}^2$,
\[
S_k=S_n+v\ \ \  \text{iff} \ \ \ S_{n-k}^{(-)}\circ T^n=v\ \ \ \text{iff}\ \ S_{n-k}\circ T^k=-v.
\]
\begin{proposition}\label{prop:moments_boundary}
	Under the assumptions of Theorem~\ref{thm: boundary range}, there exists a constant $C>0$ such that
	\begin{equation}\label{eq: first moment asymp}
	\mathbb{E}\left(\left|\partial R(n)\right|\right)\sim \frac{Cn}{\log^2(n)},\ \text{as } n\to\infty.
	\end{equation}
	In addition, there exists $M>0$ such that
	\begin{equation}\label{eq: Renyi ineq}
	\mathrm{Var}\left(\left|\partial R(n)\right|\right)\lesssim \frac{Mn^2\log\log(n)}{{\log^5(n)}}\ \  \text{as}\ n\to\infty.
	\end{equation}
\end{proposition}

\begin{proof}
It follows easily from the equality $\partial R(n) = \bigcup_{v \in
\mathds{E}_d} \partial_v R(n)$ in~\eqref{eq:partialR_union} that for any
$v\in \mathds{E}_d$ we have
\begin{equation}\label{eq:inclusion_exclusion_lowerbound}
|\partial_v R(n)| \leq |\partial R(n)|\leq \sum_{v \in \mathds{E}_d}|\partial_v R(n)|.
\end{equation}
Using the inclusion-exclusion principle and enumerating the elements of
$\mathds{E}_d$ as $v_1, \dots, v_{2d}$ we have
\begin{align*}
|\partial R(n)|
&= \sum_{V\subset \mathds{E}_d} (-1)^{|V|+1}
\left| \bigcap_{v\in V} \partial_v R(n) \right|.
%
\end{align*}
Given a collection, $V:=\{v_1, \dots, v_l\}$ say, of distinct vectors in $\mathbb{Z}^2$
we have that
$$
\bigcap_{v\in V} \partial_{v} R(n)
=R(n) \cap \bigcap_{v\in V} \left(R(n)+v \right)^c
$$
whence we can write as in the proof of
Proposition~\ref{prop:range_of_Birkhoff}
\begin{align*}
\bigg|\bigcap_{v\in V} \partial_{v} R(n)\bigg|
&=\sum_{k=1}^n \mathbf{1}_{A_{k,V}(k)}(\omega)\mathbf{1}_{B_{k,V}(n-k)}(\omega)
\end{align*}
	where
\begin{align*}
	A_{0,V}(L)= \left\{\omega\in \Omega:\ \forall 0<l<L, S^{(-)}_l\notin V \right\}, \qquad A_{k,V}(L)=T^{-k}A_{0,V}(L),
\end{align*}
	and
\begin{align*}
B_{0,V}(L)=\left\{\omega\in \Omega:\ \forall 0<l<L,\ S_l\notin V\cup\{0\} \right\}, \quad	B_{k,V}(L)=T^{-k}B_{0,V}(L).
\end{align*}
To see why notice that
	\[
A_{k}(k)\cap B_{k}(n-k) = \left\{\omega\in\Omega: \forall l\in [1,n],\ S_k\notin S_l+V,\ \forall m\in [k+1,n],\ S_m\neq S_k \right\}.
	\]

Let us first consider the case $V=\{v\}$.
Then one has
	\[
	\left|\partial_v R(n)\right|(\omega)
	=\sum_{k=1}^n \mathbf{1}_{A_{k}(k)}(\omega)\mathbf{1}_{B_{k}(n-k)}(\omega),
	\]
	where $A_{k}:=A_{k,\{v\}}$, $B_{k}:=B_{k,\{v\}}$.
	By the Markov property for the random walk, for all $1\leq k\leq n$, $A_k(k)$ and $B_k(n-k)$ are independent, thus
\begin{align*}
	\mathbb{E}\left(\left|\partial_v R(n)\right|\right)&= \sum_{k=1}^n \mathbb{P}\left(A_k(k)\right)\mathbb{P}\left(B_k(n-k)\right)\\
	&= \sum_{k=1}^{n} \mathbb{P}\Big(\forall 1\leq l\leq k, S_l\neq v\Big)\mathbb{P}\Big(\forall 1\leq l\leq n-k,\ S_l\notin \{0,v\}\Big)\\
&= \sum_{k=2}^{n-2} \frac{c_v d_v}{\log(k)\log(n-k)} (1+o(1)),
\end{align*}
as $n\to \infty$, where $c_v, d_v>0$ by Proposition~\ref{prop: return tail}.
An easy calculation, using the fact that $\log$ is slowly varying, shows that
\begin{align*}
\sum_{k=2}^{n-2} \frac{1}{\log(k)\log(n-k)}
&= \sum_{k=\lfloor n/\log^2 n\rfloor}^{n-\lfloor n/\log^2 n\rfloor}
\frac{1}{\log(k)\log(n-k)}+O\left(\frac{n}{(\log n)^3} \right)\\
&= \frac{n}{\log(n)^2}(1+o(1)) +O\left(\frac{n}{(\log n)^3} \right),
\end{align*}
where we write $\lfloor x\rfloor$ for the integer part of $x$.

Assume now $|V|> 1$. For any set $W$, the proof of Proposition~\ref{prop:
return tail} and in particular \eqref{eq:KSlimit} imply that, as $n\to\infty$
\begin{align*}
\mathbb{P}\left(\forall 1\leq k\leq n,\ S_k\notin W\right)
&=\frac{\tilde{g}_W(0)\gamma_d}{\log(n)}+o\left(\log(n)^{-1}\right)\ \text{as}\ n\to\infty,
\end{align*}
with $\gamma_d$ and $\tilde{g}_W(0)$ as defined in \eqref{eq:noreturn} and
\eqref{eq:gw} respectively.
Therefore similar arguments show that for any $V\subset \mathds{E}_d$ we have that
\begin{align*}
\rE \bigg|\bigcap_{v\in V} \partial_{v} R(n)\bigg|
\sim \gamma_d^2\, \tilde{g}_V(0) \tilde{g}_{V\cup \{0\}}(0) \, \frac{n}{\log(n)^2},
\end{align*}
where in general it is possible that the constant is zero.

Going back to $\partial R(n)$, we thus have that
\begin{align*}
\lim_{n\to\infty}\frac{\log(n)^2}{n}\rE |\partial R(n)|
&=\sum_{V\subseteq \mathds{E}_d }(-1)^{|V|+1} \gamma_d^2 \, \tilde{g}_V(0) \tilde{g}_{V\cup \{0\}}(0) \geq c_v d_v >0
\end{align*}
where the last two inequalities follow from \eqref{eq:inclusion_exclusion_lowerbound} and Proposition~\ref{prop: return tail}.
This proves \eqref{eq: first moment asymp}.

For the second part, from \eqref{eq:inclusion_exclusion_lowerbound}
$$\mathbb{E}\left(\left|\partial R(n)\right|^2\right)
= \mathbb{E}\left[\left(\sum_{v\in \mathds{E}_d}\left|\partial_v R(n)\right|\right)^2\right]
\leq C \sum_{v\in \mathds{E}_d}\mathbb{E}\left[\left|\partial_v R(n)\right|^2\right].$$
Letting $v\in \mathds{E}_d$ be arbitrary we then have
\[
\mathbb{E}\left(\left|\partial_v R(n)\right|^2\right)=\mathbb{E}\sum_{k=1}^n \left(\mathbf{1}_{A_k(k)}\mathbf{1}_ {B_k(n-k)}\right)+2\sum_{1\leq k<m\leq n}\mathbb{E}\left( \mathbf{1}_{A_k(k)}\mathbf{1}_{B_k(n-k)}\mathbf{1}_{A_m(m)}\mathbf{1}_{B_m(n-m)} \right).
\]
The first term is  equal to $\mathbb{E}\left(\left|\partial_v R(n)\right|\right)$.
For the second term, notice that for $1\leq k<m\leq n$,
	\begin{equation}\label{eq: decomposition to independent events}
	\mathbf{1}_{B_k(n-k)}\mathbf{1}_{A_m(m)}\leq \mathbf{1}_{B_k(\lfloor(m-k)/2\rfloor)}\mathbf{1}_{A_m(\lfloor(m-k)/2\rfloor)},
	\end{equation}
since  for any $k,m $, $B_k(\cdot)$, $A_m(\cdot)$ are decreasing sequences of sets.
To keep notation concise, for any integers $l,k$ we will denote $\mathbb{P}(A_k(l))$ by $\psi(l)$ and $\mathbb{P}(B_k(l))$ by $\theta(l)$, where we can drop the dependence on $k$ since $\mathbb{P}\circ T^{-1}=\mathbb{P}$.  Since for $1\leq k\leq m\leq n$ the events
	\[
	A_k(k),\, B_k\left(\lfloor(m-k)/2\rfloor \right),\, A_m\left(\lfloor(m-k)/2\rfloor \right),\, B_m(n-m),
	\]
are independent, we have
\begin{align*}
	\mathbb{E}\left(\mathbf{1}_{A_k(k)}\mathbf{1}_{B_k(n-k)}\mathbf{1}_{A_m(m)}\mathbf{1}_{B_m(n-m)}\right)\leq \psi(k)\theta\left(\frac{m-k}{2}\right)\psi\left(\frac{m-k}{2}\right)\theta(n-m).
	\end{align*}
	This shows that for $k<m$,
	\[
	\iota(k,m):=\mathbb{E}\left(\mathbf{1}_{A_k(k)}\mathbf{1}_{B_k(n-k)}\mathbf{1}_{A_m(m)}\mathbf{1}_{B_m(n-m)}\right)
    -\mathbb{E}\left(\mathbf{1}_{A_k(k)}\mathbf{1}_{B_k(n-k)}\right)\mathbb{E}\left(\mathbf{1}_{A_m(m)}\mathbf{1}_{B_m(n-m)}\right)
	\]
	is bounded from above by
	\[
	\psi(k)\theta(n-m)\left\{\psi\left(\frac{m-k}{2}\right)\left[\theta\left(\frac{m-k}
	{2}\right)-\theta(n-k)\right]+\theta(n-k)\left[\psi\left(\frac{m-k}{2}\right)-\psi(m)
	\right]\right\}
	\]
	Denote by
	\[
	D(n):=\left\{(k,m)\in [1,n]^2:\ k, n-m>\sqrt{n}\ \  \text{and}\ \  m-k\geq \frac{n}{\log^5(n)}\right\}.
	\]
	Since for $n$ large enough
	\[
	\log(n)-5\log\log(n) - \log(2) \geq \frac{\log(n)}{2},
	\]
	it follows that for all $(k,m)\in D(n)$,
	\begin{equation}\label{eq: 111}
	\psi\left(\frac{m-k}{2}\right) \sim \frac{c_v}{\log((m-k)/2)} \leq \frac{c_v}{\log(n)-5\log\log(n)-\log(2)}\leq \frac{2c_v}{\log(n)},
	\end{equation}
	and
	\begin{equation}\label{eq: 222}
	\theta\left(\frac{m-k}{2}\right) \sim \frac{d_v}{\log((m-k)/2)}\leq \frac{2d_v}{\log(n)}.
	\end{equation}
with $c_v, d_v$ as in Proposition~\ref{prop: return tail}.
Similarly,
for all $(k,m)\in D(n)$ and $n$ large enough we have that
\begin{align*}
\theta\left(\frac{m-k}{2}\right)-\theta(n-k)
&= \sum_{j=\lfloor \frac{m-k}{2}\rfloor}^{n-k} \left(\theta(j)-\theta(j+1)\right)\\
&\leq \sum_{j=\lfloor \frac{m-k}{2}\rfloor}^{n-k} \rP ^{0}\left[ S_{j+1}\in \{0,v\}, S_l \notin \{0,v\}, \forall 1\leq l\leq j\right]\\
& = \sum_{j=\lfloor \frac{m-k}{2}\rfloor}^{n-k} \sum_{w\in \{0,v\}}\rP^{0}\left[ S_{j+1}=w, S_l \notin \{0,v\}, 1\leq l\leq j\right].
\end{align*}
By Lemma~\ref{lem:Q_control}, the term in the sum is bounded by $C/(j
\log^2(j))$. Thus, since  $(m-k)/2\to\infty$ when $n\to \infty$,
\begin{align}
\theta\left(\frac{m-k}{2}\right)-\theta(n-k)
&\leq C\sum_{j=\lfloor \frac{m-k}{2}\rfloor}^{n-k} \frac{1}{j \log(j)^2}\notag\\
&\leq C\pi\gamma \int_{s=\lfloor \frac{m-k}{2}\rfloor}^{n-k} \frac{\rd s}{s \log(s)^2}\notag\\
&= C\pi\gamma \left[ \frac{1}{\log((m-k)/2)}-\frac{1}{\log(n-k)}\right]\notag\\
&\leq  C\pi \gamma \left[ \frac{1}{\log n - 5 \log \log n - \log 2}-\frac{1}{\log n} \right]\notag\\
& \leq \frac{C\log\log(n)}{\log^2(n)},\label{eq: 333}
\end{align}
and similarly
\begin{equation}\label{eq: 444}
\psi\left(\frac{m-k}{2}\right)-\psi(m)\lesssim  \frac{C\log\log(n)}{\log^2(n)}.
\end{equation}	
since for $(k,m)\in D(n)$, it holds that ${n}/{\log^5(n)}\leq m\leq n$.
	
Combining \eqref{eq: 111}, \eqref{eq: 222}, \eqref{eq: 333} and \eqref{eq: 444}, we
obtain a global constant $M>0$ such that for large $n$, for all $(k,m)\in D(n)$
\begin{equation}\label{eq: 555}
	\iota(k,m)\lesssim \frac{M\log\log(n)}{\log^3(n)}\frac{1}{\log(n-m)}\frac{1}{\log(k)}\leq 	\frac{4M\log\log(n)}{\log^5(n)}.
\end{equation}
Since $\iota(k,m)\leq 1$ for all $k,m$, we have
\begin{align*}
\sum_{1\leq k<m\leq n}\iota(k,m) &\leq \# \left\{(k,l)\in [1,n]^2\setminus D(n):\ k<l\right\}+	\sum_{(k,m)\in D(n)}\iota(k,m)\\
&\leq \left(2 n^{3/2}+\frac{n^2}{\log^5(n)}\right)+\frac{Mn^2\log\log(n)}{\log^5(n)}.
\end{align*}
This together with \eqref{eq: first moment asymp} implies \eqref{eq: Renyi ineq}.
\end{proof}
\begin{corollary}
Under the assumptions of Theorem~\ref{thm: boundary range}, there exists a
constant $C>0$ such that
\[
\lim_{n\to\infty}\frac{\left|\partial R(n)\right|}{\left(n/\log^2(n)\right)}=C,\ \  \text{in probability.}
\]
\end{corollary}

The gap between the variance of $|\partial R(n)|$ and the square of its expectation implies via Chebyshev's inequality that for all  $\epsilon>0$ there exists $C>0$ such that
\begin{equation}\label{eq: conc}
\mathbb{P}\Big(\big||\partial R(n)|-\mathbb{E}\left(|\partial R(n)|\right)\big|>\epsilon\mathbb{E}\left(|\partial R(n)|\right) \Big)\leq \frac{C\log\log(n)}{\epsilon^2 \log(n)}.
\end{equation}
The gap of order $1/\log(n)$ in the decay of these probabilities is enough to
guarantee almost sure convergence of $|\partial R_{N_k}|$ for $N_k=\exp\left(
n^a \right)$ with $a>1$. By a different method looking at \[
Z_n=\frac{\left|\partial R(n)\right|}{\mathbb{E}\left(\left|\partial
R(n)\right|\right)}-1,
\]
one can show almost sure convergence  of $Z_{N_k}$ to $0$ where $N_k$ is of the form $[\exp(n^a)]$ with $a>1/2$. This is since for $a>1/2$, $\sum_{k=1}^\infty \mathbb{E}\left(\left|Z_{N_k}\right|^2\right)<\infty$.\\
Unfortunately as $|\partial R(n)|$ is not necessarily monotone, this subsequence is too thin in order to interpolate the almost sure convergence from the subsequence to almost sure convergence along the whole sequence. To that end we  use a method from \cite{Flatto}.
\begin{defn}
	(i) Given $\delta>0$, the sequence of random variables $|\partial R(n)|$ satisfies \textit{property} $\mathbf{A}(\delta)$ if for all $\epsilon_0>0$ there exists $C=C(\epsilon_0,\delta)>0$ such that for all $n\geq 2$ and $\epsilon\geq \epsilon_0$,
	\[
	\mathbb{P}\left(|\partial R(n)|>(1+\epsilon)\mathbb{E}\left(|\partial R(n)|\right) \right)\leq \frac{C}{\epsilon^2 \log^\delta(n)}.
	\]
	(ii) Given $\delta>0$, the sequence of random variables $|\partial R(n)|$ satisfies property $\mathbf{D}(\delta)$ if for all $\epsilon_0>0$ there exists $C=C(\epsilon_0,\delta)>0$ such that for all $n\geq 2$ and $\epsilon\geq \epsilon_0$,
	\[
	\mathbb{P}\left(|\partial R(n)|<(1-\epsilon)\mathbb{E}\left(|\partial R(n)|\right) \right)\leq \frac{C}{\epsilon^2 \log^\delta(n)}.
	\]
\end{defn}

\begin{theorem}\label{thm: Flatto}[Theorem~4.2 in \cite{Flatto}]
	For all $\delta>0$, if $|\partial R(n)|$ satisfies property $\mathbf{A}(\delta)$ (respectively $\mathbf{D}(\delta)$) then it satisfies property $\mathbf{A}(4\delta/3)$ (respectively $\mathbf{D}(4\delta/3)$).
\end{theorem}
It follows from \eqref{eq: conc} that for all $0<\delta<1$, the sequence $|\partial R(n)|$ satisfies properties $\mathbf{A}(\delta)$ and $\mathbf{D}(\delta)$.
\begin{corollary}\label{cor: yes}
	For all $\delta>0$, the sequence  $|\partial R(n)|$ satisfies property $\mathbf{A}(\delta)$ and $\mathbf{D}(\delta)$.
    Consequently, taking $\delta=5$, for all $\varepsilon_0>0$ there exists $C=C(\varepsilon_0)$ such that
    for all $\epsilon >\varepsilon_0$ and $n\geq 2$,
	\begin{equation}
	\mathbb{P}\left(\left||\partial R(n)|-\mathbb{E}\left(|\partial R(n)|\right)\right|>\epsilon\mathbb{E}\left(|\partial R(n)|\right) \right)\leq \frac{C}{\epsilon^2 \log^5(n)}.
	\end{equation}
\end{corollary}
The proof of Theorem \ref{thm: Flatto} is similar to the proof of \cite[thm. 4.2]{Flatto},  hence it is postponed to the appendix.
\begin{proof}[Proof of Thm.~\ref{thm: boundary range}]
	By \eqref{eq: first moment asymp} it is enough to show that
	\[
	Z_n:=\left|\frac{\left|\partial R(n)\right|}{\mathbb{E}\left(\left|\partial R(n)\right|\right)}-1\right|\xrightarrow[n\to\infty]{a.s.}0.
	\]
	Let $\epsilon>0$ and write $N_k:=\left\lfloor\exp\left(\sqrt[4]{k}\right)\right\rfloor$, where $\lfloor\cdot\rfloor$
    is the floor function. By Corollary \ref{cor: yes} there exists $C>0$ such that for all $n\geq 2$,
	\[
	\mathbb{P}\left(Z_{N_k}>\epsilon \right)\leq \frac{C}{\epsilon^2 \log^5(N_k)}\leq \frac{C}{k^{5/4}}.
	\]
	It follows from the Borel-Cantelli lemma that for almost every $w\in\Omega$,
	\[
	\limsup_{k\to\infty}Z_{N_k}\leq \epsilon
	\]
	As $Z_n\geq 0$ and $\epsilon$ is arbitrary it follows that for almost every $w\in\Omega$,
	\[
	\lim_{n\to\infty}Z_{N_k}=0.
	\]
	
Now for a general $n\in\mathbb{N}$ there exists a unique $m=m(n)\in\mathbb{N}$ such that $N_m\leq n <N_{m+1}$. Since $e^x-1\leq 2x$ for all $0\leq x\leq 1$ and  for all $m$ large $\sqrt[4]{m+1}-\sqrt[4]{m}\leq 1/3m^{3/4}$, it follows that
\begin{equation}\label{eq: N_m}
  N_{m+1}-N_m=\exp\left(\sqrt[4]{m}\right)\left[\exp\left(\sqrt[4]{m+1}-\sqrt[4]{m}\right)-1\right]\leq m^{-3/4}N_m.
\end{equation}

In particular, $N_{m+1}/N_m \to 1$. By~\eqref{eq: first moment asymp},
$\mathbb{E}( |\partial R(n)|) \sim C n/(\log^2(n))$ is regularly varying.
Therefore, since $N_m\leq n \leq N_{m+1}$,
\begin{equation}
\label{eq:E_ratio}
\mathbb{E}(|\partial R(n)|) / \mathbb{E}(|\partial R_{N_m}|) \to 1.
\end{equation}

From \eqref{eq: N_m} and the trivial bound
\[
  \left|\partial R(n)\right|-2dm\leq \left|\partial R(n+m)\right|\leq \left|\partial R(n)\right|+2dm
\]
for all $m$ large enough and $N_m\leq n\leq N_{m+1}$, we have
using~\eqref{eq: N_m}
\begin{equation*}
  \frac{\bigl||\partial R(n)| - |\partial R_{N_m}|\bigr|}{\mathbb{E}(|\partial R_{N_m}|)}
  \leq \frac{2d (N_{m+1}-N_m)}{\mathbb{E}(|\partial R_{N_m}|)}
  \leq \frac{2d m^{-3/4}N_m}{C N_m/ \log^2(N_m)}
  \lesssim \frac{m^{1/2}}{m^{3/4}} \to 0.
\end{equation*}
Since $|\partial R_{N_m}|/\mathbb{E}(|\partial R_{N_m}|)$ tends almost surely
to $1$, we obtain that $|\partial R(n)|/\mathbb{E}(|\partial R_{N_m}|)$ also
tends to $1$. Together with~\eqref{eq:E_ratio}, this gives $|\partial
R(n)|/\mathbb{E}(|\partial R(n)|) \to 1$ as required.
\end{proof}
\section{The Range of random walks in the domain of attraction of the symmetric \texorpdfstring{$\alpha$}{alpha}
stable distribution for \texorpdfstring{$1<\alpha\leq
2$}{1<alpha<2}}\label{sec: Folner} Let $\{X_n\}_{n=-\infty}^\infty$ be as
sequence of i.i.d., centered, $\mathbb{Z}$-valued random variables, and
define the two-sided random walk $\{S_n\}_{n\in \mathbb{Z}}$ as follows,
$S_0=0$, and for $n\geq 1$ let $S_n=X_0+\dots +X_{n-1}$ and
\smash{$S_{n}^{(-)} = -X_{-1}-\cdots-X_{-n}$}. We assume that  $\{S_n\}_n$ is
aperiodic in the sense of Section~\ref{sec: boundary Z^2} and that the random
variables $\{X_i\}_i$ belong to the domain of attraction of a nondegenerate,
symmetric, $\alpha$-stable distribution with $1<\alpha\leq 2$. That is, there
exists a positive slowly varying function $L:\mathbb{R}_+\to \mathbb{R}_+$
such that,
$$
Y_n:=\frac{S_n}{n^{\frac{1}{\alpha}}L(n)}\xrightarrow[n\to\infty]{d}Z_\alpha,
$$
where $Z_\alpha$ is a real  random variable with characteristic function $\mathbb{E}(e^{itZ_\alpha})=e^{-|t|^\alpha}$.
By Levy's continuity theorem, writing $\phi(t):=\mathbb{E}(e^{itX_1})$, we see that for all $t>0$,
\[
\rE \left(e^{itY_n}\right)=\phi\left(\frac{t}{n^{\frac{1}{\alpha}}L(n)} \right)^n\xrightarrow[n\to\infty]{} e^{-|t|^\alpha}.
\]
From this and a Tauberian theorem it follows that, see e.g.\ \cite[Theorem~2.6.5]{Ibragimov1975independent},
\[
\phi(t)=1-|t|^\alpha L(1/|t|)\left[1+o\left(1\right)\right], \quad t \to 0.
\]
From now on $L$ will denote a positive slowly varying function which can change from line to line.
If $A\subset \mathbb{Z}$ is a finite subset we define
$$r_n(x,A):=\mathbb{P}^x\left(S_k\notin A, \,\, 1\leq k\leq n\right),$$
and write $r_n:=r_n(0,\{0\})$.

We will need the following result.
\begin{lemma}\label{lem:tauberian}
	Assume that $\{X_i\}_{i \in \mathbb{Z}}$ are aperiodic i.i.d., $\mathbb{Z}$-valued random variables such that
	$1-\phi(t)$ is regularly varying of index $\alpha$ as $t\to 0$, for $\alpha \in (1,2)$.
	Then $r_n$ is regularly varying of index $1/\alpha -1$ as $n\to \infty$. In addition, there exists a positive slowly varying
	function $L$ such that, for any $x\in \mathbb{Z}$ and any finite nonempty $A\subseteq \mathbb{Z}$,
	\begin{equation}\label{eq:general_bound}
	\lim_{n\to \infty} n^{1-1/\alpha} L(n) r_n(x,A) = C(x,A),
	\end{equation}
	for some constant $C(x,A)\geq 0$.
\end{lemma}
\begin{proof}
When $\lim_{s\to \infty} L(s)$ exists this result is \cite[Theorem~8]{Kes}.
So we will consider the case where $L(\cdot)$ does not have a limit at
infinity. Similarly to the proof of \cite[Theorem~8]{Kes} we define for
$\lambda \in [0,1)$
\begin{align*}
U(\lambda) &:= \sum_{n=0}^\infty \lambda^n \rP(S_n=0) = \frac{1}{2\pi} \int_{-\pi}^{\pi} \frac{\rd t}{1-\lambda \phi(t)},\\
R(\lambda) &:= \sum_{n=0}^\infty \lambda^n r_n = (1-\lambda)^{-1} U(\lambda)^{-1}.
\end{align*}
First we will study the asymptotic behaviour of $U(\lambda)$ as $\lambda \to 1$.

Notice that since $|t|^\alpha L(1/t)$ is regularly varying at the origin, by
\cite[Theorem~1.5.3]{Bingham_89} there exists a monotone, $\alpha$-regularly
varying function $g(t)$, such that $g(t) \sim C|t|^\alpha L(1/t)$ as $t \to
0$, for some $C>0$. Therefore $\phi(t)=1-g(t)\left[1+o\left(1\right)\right]$,
as $t\to 0$. Next, we use aperiodicity to concentrate on the behaviour around
the origin. In particular for any $\epsilon>0$, there exists a constant
$C(\epsilon)>0$ such that $|\phi(t)|<1-C(\epsilon)$ for all $|t|>\epsilon$.
The main contribution will come from $|t|<\epsilon$. We claim that the
function
\begin{equation*}
  \tilde{U}_\epsilon(\delta) := \int_{0}^{\epsilon} \frac{\rd t}{\delta +g(t)}
\end{equation*}
is slowly varying at $0$ with index $1/\alpha-1$, the asymptotics not
depending on $\epsilon$. Let us first see how this claim implies the result.
We write
\begin{multline*}
  U(\lambda) = \frac{\tilde U_\epsilon(1-\lambda)}{\pi}
  + \frac{1}{2\pi}\left(\int_{0}^{\epsilon} \frac{\rd t}{1-\lambda\phi(t)} - \tilde U_\epsilon(1-\lambda)\right)
  \\
  + \frac{1}{2\pi}\left(\int_{-\epsilon}^{0} \frac{\rd t}{1-\lambda\phi(t)} - \tilde U_\epsilon(1-\lambda)\right)
  + \frac{1}{2\pi} \int_{\pi>|t|>\epsilon} \frac{\rd t}{1-\lambda \phi(t)}.
\end{multline*}
We claim that the term $\tilde U_\epsilon(1-\lambda)$ on the right dominates
the other ones when $\lambda$ tends to $1$. First, it dominates the last one
as $\tilde U_\epsilon(1-\lambda)$ tends to infinity while the last term
remains bounded. The other two terms are similar, let us handle the first
one. We write the difference as
\begin{equation*}
  \int_{0}^{\epsilon} \frac{\rd t}{1-\lambda\phi(t)} - \tilde U_\epsilon(1-\lambda)
  = \int_{0}^{\epsilon} \left(\frac{1}{1-\lambda\phi(t)} - \frac{1}{1-\lambda + g(t)}\right) \rd t
  = \int_{0}^{\epsilon} \frac{g(t)-\lambda (1-\phi(t))}{(1-\lambda\phi(t))(1-\lambda + g(t))} \rd t.
\end{equation*}
When $\lambda$ is close enough to $1$, the numerator is bounded by
$\eta(\epsilon) g(t)$, where $\eta$ tends to $0$ with $\epsilon$. Canceling
with the first factor of the denominator, it follows that this term is
bounded by $\eta(\epsilon) \tilde U_\epsilon(1-\lambda)$, as desired.

The regular variation of $\tilde U_\epsilon$ then implies that $U(\lambda)$
is regularly varying of index $1/\alpha-1$. Thus $R(\lambda) \sim
C(1-\lambda)^{-1/\alpha} L(1/(1-\lambda))$, and since $r_n$ is monotone it
follows that $r_n\sim c n^{1/\alpha -1} L(n)$ for some constants $c, C>0$.
Having established this, \eqref{eq:general_bound} follows from
\cite[Theorem~4a]{KS63}.

\medskip

It remains to prove the claim that $\tilde U_\epsilon(\delta)$ is slowly
varying at $0$ with index $1/\alpha-1$. Since $g$ is $\alpha$-regularly
varying and monotone,  its inverse $f(u):= g^{-1}(u)$ will be regularly
varying of index $1/\alpha$ and monotone. Next, letting $t=g^{-1}(z)$ we have
\begin{align*}
\tilde{U}_\epsilon(\delta)
&= \int_{z=0}^{g(\epsilon)} \frac{\rd g^{-1}(z)}{\delta +z}= \int_{z=0}^{\infty} \frac{\rd g_\epsilon^{-1}(z)}{\delta +z},
\end{align*}
interpreting the integral in the Stieltjes sense, and letting
$g_\epsilon^{-1}(s):=g^{-1}(s)$ for $s< g(\epsilon)$ and
$g_\epsilon^{-1}(s):=g^{-1}(g(\epsilon))$ for $s \geq g(\epsilon)$. With this
definition it is clear that $g^{-1}_\epsilon$ is also monotone and that
$g^{-1}_\epsilon(s)\sim g^{-1}(s)$ as $s\to 0+$. In particular its behaviour
near the origin is independent of $\epsilon$.

The rest is fairly similar to the proof of \cite[Theorem~1.7.4]{Bingham_89}. First notice that
\begin{align*}
\tilde{U}_\epsilon(\delta)
&= \int_{z=0}^{\infty} \frac{\rd g_\epsilon^{-1}(z)}{\delta +z}\\
&= \int_{z=0}^{\infty}\int_{u=0}^\infty \re^{-(\delta+z)u} \rd g^{-1}_\epsilon(z)\rd u\\
&=\int_{u=0}^\infty \re^{-\delta u} \int_{z=0}^{\infty}\re^{-z u} \rd g^{-1}_\epsilon(z)\rd u\\
&=\int_{u=0}^\infty \re^{-\delta u} V_\epsilon(u) \rd u,
\end{align*}
where
$$V_\epsilon(u):= \int_{z=0}^{\infty}\re^{-z u} \rd g_\epsilon^{-1}(z).$$
By~\cite[Theorem~1.7.1']{Bingham_89} we have that since $g_\epsilon^{-1}(z)$
is regularly varying of index $1/\alpha$ at the origin, we have that
$V_\epsilon(u)$ is regularly varying of index $-1/\alpha$ as $u\to\infty$. In
turn this implies that
$$W_\epsilon(u):= \int_0^u V_\epsilon(s)\rd s,$$
is regularly varying of index $1-1/\alpha$ as $u\to \infty$ and since
\begin{align*}
\tilde{U}_\epsilon(\delta)
&=\int_{u=0}^\infty \re^{-\delta u} V_\epsilon(u) \rd u
=\int_{u=0}^\infty \re^{-\delta u} W_\epsilon(\rd u),
\end{align*}
by Karamata's Tauberian theorem \cite[1.7.1]{Bingham_89} we have that
$\tilde{U}_\epsilon(\delta)$ is regularly varying of index $1/\alpha -1$ as
$\delta \to 0$.
\end{proof}
\begin{remark}
An alternative approach is possible using
\begin{equation}
 	\left|\mathbb{P}\left(S_n=m\right)- \mathbb{P}\left(S_n=0\right)\right| =O\left( \frac{1}{a_n^2}\right) ,\ \ as\ n\to\infty
\end{equation}
where $m\in \mathbb{Z}$ and $a_n$ is a $1/\alpha$ regularly varying sequence, given in \cite{GJP84}.
\end{remark}

\begin{theorem}\label{thm: alpha stable}
	Let $S_n$ be an aperiodic random walk satisfying \eqref{eq:general_bound} with $\alpha\in (1,2]$, then for all $\epsilon>0$, almost surely
	\[
	\frac{\left|\partial R_n\right|}{\left|R_n\right|}= o\left(n^{\frac{1}{\alpha}-1+\epsilon}\right).
	\]
\end{theorem}
\begin{remark}
	This theorem, in the case of i.i.d. random variables with $\mathbb{E}\left(X_1\right)=0$ and $\mathbb{E}\left(X_1^2\right)=D<\infty$ gives the rate $o(n^{p})$ for every $p<\frac{1}{2}$. We claim that this is the optimal rate in the polynomial exponent. Indeed, for the simple random walk on $\mathbb{Z}$, $\left|\partial R_n\right|=2$ for all $n$, and $\frac{R_n}{\sqrt{n}}$ converges in distribution to $Y=\mathrm{Leb}\left(W[0,1]\right)$ where $W$ is the Wiener process and $W[0,1]=\left\{W_s:\ s\in [0,1]\right\}$. As the limiting random variable $Y$ is unbounded, Theorem \ref{thm: alpha stable} is optimal in the polynomial exponent.
\end{remark}

\begin{proposition}\label{prop: moment}
	Let $S_n$ be an aperiodic random walk satisfying \eqref{eq:general_bound} with $\alpha\in (1,2]$, then for all $\epsilon>0$ and $k\in\mathbb{N}$
	there exists $M>0$ such that for all $x\in\mathbb{Z}\setminus \left\{0\right\}$,
	\[
	\mathbb{E}\left(\left|R_n\setminus\left(R_n-x\right))\right|^k\right)
	\leq M n^{k\left(\frac{2}{\alpha}-1+\epsilon\right)}.
	\]
\end{proposition}
\begin{proof}
Let $\epsilon>0$ and $j\in\mathbb{Z}\setminus\{0\}$. In the course of the proof $C$ will denote a global positive constant whose value can change (increase) from line to line. As before we have the bound
\[
\left|R_n\setminus \left(R_n-x\right)\right|\leq \sum_{j=1}^n \mathbf{1}_{A_n(j)}\circ T^j
\]
where
\[
A_n(m)=\left\{w\in \Omega: S_j\neq x,\ \ \forall j\in [-m,n-m] \right\} = D^{-}_{x}(m) \cap D^{+}_{x}(n-m),
\]
where for $y\in \mathbb{Z}$ and $n\geq 0$ we define
$$D^{-}_{y}(n):= \{\omega\in \Omega: S_j^{(-)}\neq y, \, j\in [1,n]\}, \quad  D^{+}_{y}(n):= \{\omega\in \Omega: S_j^{}\neq y, \, j\in [1,n]\}.$$
Note that $D_x^-$ and $D_x^+$ are independent, thus
\[
	\mathbb{P}\left(A_n(m)\right)=
	\mathbb{P} \left( D^{-}_{x}(m) \right) \mathbb{P} \left( D^{+}_{x}(n-m) \right).
\]
As $Z_\alpha$ is symmetric,  $X_1$ is in the domain of attraction of
$Z_\alpha$ if and only if $-X_1$ is in the domain of attraction of
$Z_\alpha$.  This together with Lemma~\ref{lem:tauberian} implies that there
exist positive, slowly varying functions $L_+$, $L_{-}$  such that
\[
\mathbb{P}\left(D^{\pm}_x(n)\right)\leq \frac{L_{\pm}(n)}{n^{1-1/\alpha} },
\]
for $n\geq 1$, potentially absorbing any constants in the functions $L_\pm(\cdot)$.

We will now show that for all $k\in\mathbb{N}$ there exists $C>0$ such that,
\[
\mathbb{E}\left(\left|R_n\setminus\left(R_n-x\right))\right|^k\right)\leq C\left[ n^{k\left(\frac{2}{\alpha}-1+\epsilon\right)}+5k^2\mathbb{E}\left(\left|R_n\setminus\left(R_n-x\right)\right|^{k-1}\right)\right].
\]
The latter inequality proves the proposition. Writing
\[
\Delta_k(n):=\left\{\left(m_1,m_2,..,m_k\right)\in \left\{1,\ldots,n\right\}^k :\ \forall 1<l\leq k,\ m_l-m_{l-1}\geq 3  \right\},
\]
and
\[
\Lambda_k(n):=\left\{ \left(m_1,m_2,..,m_k\right)\in \left\{1,\ldots,n\right\}^k :\ \exists 1<l<l'\leq k,\ \left|m_l-m_{l'}\right| \leq 2  \right\},
\]
Then
\[
\left|R_n\setminus\left(R_n-x\right))\right|^k\leq k! \sum_{\left(m_1,...,m_k\right)\in\Delta_k(n)}\prod _{l=1}^k\mathbf{1}_{A_{m_l}\left(n\right)}(k)\circ T^{m_l}+k!\sum_{\Lambda_k(n)}\prod _{l=1}^k\mathbf{1}_{A_{m_l}\left(n\right)}\circ T^{m_l}
\]
For every $\left(m_1,...,m_k\right)\in \Lambda_k(n)$ there exists a minimal
$1\leq \mathbf{l}<k$ such that there exists $l'>\mathbf{l}$ with
$\left|m_{\mathbf{l}}-m_{l'}\right|\leq 2$. Since
\[
\prod _{l=1}^k\mathbf{1}_{A_{m_l}\left(n\right)}\circ T^{m_l}\leq \prod _{ l\in\{1,..k\}\setminus\{\mathbf{l}\}}\mathbf{1}_{A_{m_l}\left(n\right)}\circ T^{m_l},
\]
by a simple counting argument, we can see that
\[
\mathbb{E}\left(\sum_{\Lambda_k(n)}\prod _{l=1}^k \mathbf{1}_{A_{m_l}\left(n\right)}\circ T^{m_l}\right)\leq 10\binom{k}{2}\mathbb{E}\left(\left|R_n\setminus\left(R_n-x\right))\right|^{k-1}\right).
\]
It remains to bound the other term. For $\left(m_1,...,m_k\right)\in
\Delta_k(n)$, letting $q_1=m_1$, $q_{k+1}=n-m_k$ and for $2\leq j\leq k$,
$q_j=\lfloor(m_j-m_{j-1})/2]\rfloor$, where $\lfloor x\rfloor $ denotes the
integer part of $x$, we deduce the inequality
\[
\prod _{l=1}^k\mathbf{1}_{A_{m_l}\left(n\right)}\circ T^{m_l}
\leq  \mathbf{1}_{D_x^-\left(m_{1}\right)}\circ T^{m_1} \left(\prod_{l=2}^{k}  \big[\mathbf{1}_{D_x^+\left(q_l\right)}\circ T^{m_{l-1}} \mathbf{1}_{D_x^-\left(q_{l}\right)}\circ T^{m_l}\big] \right)
\mathbf{1}_{D_x^+\left(q_{k+1}\right)}\circ T^{m_{k}}
\]
by replacing the restriction that $S_j+x \neq S_{m_l}$ for $j\in [1,n]$ by
the weaker one that $S_j+x \neq S_{m_l}$ for $j\in [m_{l-1}+ q_{l},m_{l}+
q_{l+1}]$; see Figure~\ref{fig1} where the random walk started at the marked
point is constrained to not visit $x$ for the time indicated by the arrows.
\begin{figure}[t]
	\includegraphics[scale=.6]{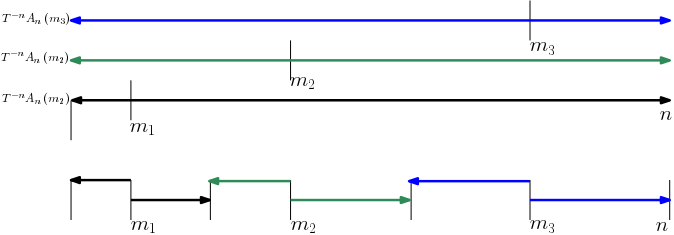}
	\caption{\label{fig1}}
\end{figure}
The bound above is essentially a product of events that depend on
non-overlapping sequences of the random variables $\{X_j\}_{j\in \mathbb{Z}}$
and therefore by independence we have that
\begin{align*}
\mathbb{E}\left( \prod _{l=1}^k\mathbf{1}_{A_{m_l}\left(n\right)}\circ T^{m_l}\right)
&\leq \mathbb{P}\left\{D_x^-\left(q_{1}\right)\right\}
\prod _{l=2}^k \left[\mathbb{P}\big\{D_x^{+}\left(q_l\right)\big\} \mathbb{P}\left\{D_x^-\left(q_{l}\right)\right\}\right] \mathbb{P}\big\{D_x^+\left(q_{k+1}\right)\big\}\\
&\leq C\prod_{l=2}^{k+1} \frac{L_{-} (q_{l-1})}{(q_{l-1})^{1-1/\alpha}}\frac{L_{+} (q_{l})}{(q_{l})^{1-1/\alpha}}.
\end{align*}
In particular since $L_{\pm}(\cdot)$ are slowly varying, for any $\epsilon>0$ we can find a positive constant $C>0$ such that $L_{\pm}(n) \leq C n^{\epsilon}$ and thus
\begin{align*}
\mathbb{E}\left( \prod _{l=1}^k\mathbf{1}_{A_{m_l}\left(n\right)}\circ T^{m_l}\right)
&\leq C\prod_{l=2}^{k+1} \frac{C}{(q_{l-1} q_l)^{1-1/\alpha-\epsilon}}\\
&= C \left[q_1 q_{k+1} \right]^{1/\alpha -1 +\epsilon} \prod_{l=2}^k (q_{l})^{2/\alpha -2 +2\epsilon}.
\end{align*}
Therefore,
\begin{align*}
\mathbb{E}\left(\sum_{\Delta_k(n)}\prod _{l=1}^k\mathbf{1}_{A_{m_l}\left(n\right)}\circ T^k\right)
&\leq C \sum_{q_1,\dotsc,q_{k+1}} [q_1 q_{k+1}] ^{\frac{1}{\alpha}-1+\epsilon}\prod _{l=2}^{k} q_l^{\frac{2}{\alpha}-2+2\epsilon}.
\end{align*}
The sum is restricted to the values of $q_i$ that can be produced by the
above process. They satisfy $q_1+2q_2+\dotsb+2q_k+q_{k+1} \leq n$ and
$q_1+2q_2+\dotsb+2q_k+q_{k+1} \geq n-k$, because of the integer parts in the
definition of $q_j$. The convolution of two sequences which are
$O(n^{\alpha_1})$ and $O(n^{\alpha_2})$ with $\alpha_1,\alpha_2>-1$ is
$O(n^{1+\alpha_1+\alpha_2})$, and the exponent is again $>-1$. In the above
sum, the exponents are respectively $1/\alpha-1+\epsilon$ and
$2/\alpha-2+2\epsilon$. As $\alpha \in (1,2]$, they are $>-1$. Therefore, one
can apply recursively this estimate on convolution of sequences, and deduce
that the above sum is $O(n^\rho)$ for
\begin{equation*}
  \rho = k + 2 \cdot \left(\frac{1}{\alpha} - 1+\epsilon\right) + (k-1) \cdot \left(\frac{2}{\alpha}-2 + 2 \epsilon\right)
  = k \cdot\left(\frac{2}{\alpha}-1 + 2\epsilon\right).
\end{equation*}
The result follows with $\epsilon'=2\epsilon$.
\end{proof}

\begin{proposition}\label{prop: range}
Let $S_n$ be a random walk satisfying \eqref{eq:general_bound} with
$\alpha\in (1,2]$. Then for all $\epsilon>0$, almost surely
\[
\lim_{n\to\infty}\frac{\left|R_n\right|}{n^{\frac{1}{\alpha}-\epsilon}}=\infty.
\]
\end{proposition}
\begin{proof}
Let $\epsilon>0$. Le Gall and Rosen have shown in \cite{LeR} that there exists a $1/\alpha$-regularly varying sequence $a_n$ such that
\begin{equation}\label{Ro-Legall}
\frac{\left|R_n\right|}{a_n}\xrightarrow[n\to\infty]{dist.} \mathrm{Leb}_\mathbb{R}\left(W_\alpha[0,1]\right),
\end{equation}
where $W_\alpha[0,1]$ is the range of the symmetric $\alpha$-stable Levy motion up to time $1$. It is well known, see for example \cite{B64}, that the occupation measure of a one-dimensional $\alpha$-stable processes, for $\alpha>1$, defined by
\[
\mu (A):=\int_0^1 \mathbf{1}_A\circ W_sds
\]
is almost surely absolutely continuous with respect to Lebesgue measure. \footnote{Equivalently almost surely possess a continuous local time $x\mapsto L_\alpha(1,x)$} As $\mu(W_\alpha[0,1])=1$ this implies that
\begin{equation}\label{eq:barlow}
\rP\Big\{\mathrm{Leb}_\mathbb{R}\left(W_\alpha[0,1]\right)>0\Big\}=1.
\end{equation}

Since $a_n$ is $1/\alpha$-regularly varying, setting $t_n:=\lfloor n^\kappa
\rfloor$, where $\kappa>0$, we have that $a_{t_n}$ is
$\kappa/\alpha$-regularly varying. Let $\kappa = 1-\alpha\epsilon/2$, so that
$\kappa/\alpha>1/\alpha-\epsilon$. Then for $n$ large enough we have
$a_{t_n}>n^{1/\alpha -\epsilon}$. Decompose the interval $[0,n]\cap
\mathbb{Z}$ into sub intervals $[jt_n,(j+1)t_n]$, $j=0,1,..,\left\lfloor
{n}/{t_n}\right\rfloor$ of length $t_n$ plus perhaps a remainder interval
which will be ignored. Writing $R(n,m)=\left\{S_{n+1},..,S_m\right\}$, for
$\delta>0$ and all $n$ large enough we have that
\begin{align*}
\mathbb{P}\left(\left|R_n\right|<\delta n^{\frac{1}{\alpha}-\epsilon}\right)
&\leq \mathbb{P}\Big(\max_{j\in\{1,..,\left\lfloor \frac{n}{t_n} \right\rfloor\} } \left|R\left(jt_n,(j+1)t_n\right)\right|< \delta a_{t_n} \Big)\\
&\leq  \mathbb{P}\left(\left|R_{t_n}\right| < \delta a_{t_n}\right)^{\left\lfloor \frac{n}{t_n}\right\rfloor}.
\end{align*}
By \eqref{eq:barlow} we can choose $\delta>0$ small enough so that
$\mathbb{P}\left(\mathrm{Leb}_\mathbb{R}\left(W_\alpha[0,1]\right)<\delta\right)<1$ and by \eqref{Ro-Legall} we can choose $N_0$ large enough so that for all $n\geq N_0$ we have
\[
\mathbb{P}\left(\left|R_{t_n}\right|< \delta a_{t_n}\right)\leq \rho <1.
\]
Therefore for all $n>N_0$, from the definition of $t_n$ and the above it follows that
\[
\mathbb{P}\left(\left|R_n\right|<\delta n^{\frac{1}{\alpha}-\epsilon}\right) \leq \rho^{\left\lfloor \frac{n}{t_n}\right\rfloor},
\]
and since $n/t_n \sim n^{\kappa'}$ where $\kappa'=\epsilon \alpha /2$, an application of the Borel-Cantelli Lemma shows that
\[
\varliminf_{n\to\infty}\frac{\left|R_n\right|}{n^{\frac{1}{\alpha}-\epsilon}}\geq \delta.
\]
As $\epsilon >0$ is arbitrary the conclusion follows.
\end{proof}
\begin{proof}[Proof of Theorem \ref{thm: alpha stable}]
Fix $\delta>0$. By Proposition \ref{prop: range}, it remains to show that
\begin{equation}\label{eq: bdry}
\lim_{n\to\infty}\frac{\left|\partial R_n\right|}{n^{\frac{2}{\alpha}-1+\delta}}=0.
\end{equation}
First note that
\[ |\partial R_n|=|R_n\setminus (R_n-1)|+| R_n\setminus (R_n+1)|=V_n(1)+V_n(2).
\]
Let $i\in\{1,2\}$. By Proposition \ref{prop: moment} for any $\epsilon>0$ and
$k\in \mathbb{N}$ there exists $M>0$ such that for all $n\in\mathbb{N}$ and
$i\in \{1,2\}$,
\[
\mathbb{E}\left(\left(\frac{V_n(i)}{n^{\frac{2}{\alpha}-1+\delta}}\right)^k\right)\leq M n^{k(\epsilon-\delta)}.
\]
Therefore, choosing $\epsilon<\delta$, there exist $k\in\mathbb{N}$ and $M>0$
such that for all $n\in\mathbb{N}$ and $i\in \{1,2\}$,
\[
\mathbb{E}\left(\left(\frac{V_n(i)}{n^{\frac{2}{\alpha}-1+\delta}}\right)^k\right)\leq M n^{-2}.
\]
A standard use of Markov's inequality and the Borel Cantelli Lemma shows that for any $\delta>0$
almost surely
\begin{equation}\label{eq: tedious}
\lim_{n\to\infty}\frac{V_n(1)}{n^{\frac{2}{\alpha}-1+\delta}}=\lim_{n\to\infty}\frac{V_n(2)}{n^{\frac{2}{\alpha}-1+\delta}}=0.
\end{equation}
Therefore we have that almost surely for all $\delta>0$ small enough
$$\lim_{n\to\infty}\frac{\left|\partial R_n\right|}{\left|R_n\right| n^{\frac{1}{\alpha}-1 + \delta}}=\lim_{n\to\infty}\frac{\left|\partial R_n\right|/n^{\frac{2}{\alpha} -1 + \delta/2}}{\left|R_n\right|/n^{\frac{1}{\alpha} -\delta/2}}=0,
$$
which proves the theorem.
\end{proof}
\appendix

\section{Proof of Theorem \ref{thm: Flatto} via Flatto's inequality enhancement procedure}
Assume that $\mathbf{A}(\delta)$ holds for $|\partial R(n)|$. Fix $\epsilon_0>0$ and denote by $\kappa>0$ the unique constant, see Proposition~\ref{prop:moments_boundary}, so that
\[
\mathbb{E}\left(|\partial R(n)|\right) \sim \frac{\kappa n}{\log^2(n)}\ \text{as}\ n\to\infty.
\]
For $n\in \mathbb{N}$ we will write $N=N(n)=\left\lfloor
\log^{\delta/3}(n)\right\rfloor$. For $1\leq i\leq N$, write $n_i=\lfloor
ni/N \rfloor$ and divide the range $R(n)$ into $N$-blocks,
\[
X_{n,i}:=\left\{S_{n_{i-1}+1},S_{n_{i-1}+2},\ldots,S_{n_i}\right\},\ \ 1\leq i\leq N.
\]
As before
\[
\partial_v  X_{n,i}:=X_{n,i}\setminus \left\{X_{n,i}+v\right\}, \qquad
\partial X_{n,i} = \bigcup_{v\in\mathds{E}_d} \partial_v X_{n,i}.
\]
Clearly
\[
|\partial R(n)|\leq \sum_{i=1}^N |\partial X_{n,i}|.
\]
Let $\epsilon\geq \epsilon_0$ and set
\[
A_i:=\left\{\omega\in \Omega:\  |\partial  X_{n,i}|\geq \left(1+\frac{\epsilon}{2}\right)\frac{\kappa n}{N\log^2(n)} \right\}
\]
and
\[
B_i:=\left\{\omega\in \Omega:\  |\partial  X_{n,i}|\geq \left(1+\frac{\epsilon N}{2}\right)\frac{\kappa n}{N\log^2(n)} \right\},
\]
for $1\leq i\leq N$. Then a simple combinatorial argument (see equation (4.9)
in \cite{Flatto}) gives,
\begin{equation}\label{eq: 0th Flatto}
\left[|\partial R(n)|>(1+\epsilon) \frac{\kappa n}{\log^2(n)}\right]\subset \left(\bigcup_{1\leq i<j\leq N} A_i\cap A_j\right)\cup \left(\bigcup_{i=1}^N B_i\right)
\end{equation}
First we estimate $\mathbb{P}\left(A_i\right)$, $\mathbb{P}\left(B_i\right)$
from above. Writing  $m_i=m_i(n)=n_{i}-n_{i-1}$ for $1\leq i\leq N$,
$|\partial X_{n,i}|$ is equal in  distribution to $\left|\partial
R\left(m_i\right)\right|$. In addition,
\[
\lim_{n\to \infty } \frac{m_i/\log^2\left(m_i\right)}{n/(N\log^2(n))}=1
\]
uniformly in $1\leq i\leq N$. In addition, by \eqref{eq: first moment asymp}, since $m_i\to \infty$ as $n\to \infty$,
\[
\lim_{n\to\infty}\frac{\mathbb{E}\left(\left|\partial  R({m_i})\right|\right)}{\kappa m_i/\log^2\left(m_i\right)}=1.
\]
Consequently there exists $\mathbf{n}=\mathbf{n}(\epsilon_0)$, such that for all $\epsilon>\epsilon_0$ and $n>\mathbf{n}$,
\begin{equation}\label{eq: 1st Flatto}
A_i \subset \left[ |\partial  X_{n,i}|\geq \left(1+\frac{\epsilon}{3}\right)\frac{\kappa m_i}{\log^2(m_i)}\right]\subset \left[ |\partial  X_{n,i}|\geq \left(1+\frac{\epsilon}{4}\right)\mathbb{E}\left(\left|\partial R(m_i)\right|\right)\right].
\end{equation}
We deduce that for all $\epsilon>\epsilon_0$, $n>\mathbf{n}$ and $1\leq i\leq N$,
\begin{align*}
\mathbb{P}\left(A_i\right)&\leq  \mathbb{P}\left(|\partial  X_{n,i}|>\left(1+\frac{\epsilon}{4}\right)\mathbb{E}\left(\left|\partial R({m_i})\right|\right)\right) \\
&= \mathbb{P}\left(|\partial  R(m_i)|>\left(1+\frac{\epsilon}{4}\right)\mathbb{E}\left(\left|\partial  R(m_i)\right|\right)\right),\ \ \ \text{by property } \mathbf{A}(\delta)\\
&\leq \frac{16C\left(\epsilon_0/4,\delta\right)}{\epsilon^2\log^\delta \left(m_i\right)}.
\end{align*}
Finally, as $\log^\delta (m_i)\sim \log^\delta(n)$ as $n\to \infty$, we can
enlarge $\mathbf{n}$ such that for all $\epsilon>\epsilon_0$, and
$n>\mathbf{n}$,
\begin{equation}\label{eq: 2nd Flatto}
\mathbb{P}\left(A_i\right)\leq \frac{32C\left(\epsilon_0/4,\delta\right)}{\epsilon^2\log^\delta \left(n\right)}
\leq \frac{32C\left(\epsilon_0/4,\delta\right)}{\epsilon_0 \epsilon \log^\delta \left(n\right)}
\end{equation}
To bound $\mathbb{P}(B_i)$ from above notice that by similar considerations as in \eqref{eq: 1st Flatto}, for all $\epsilon>\epsilon_0$ and $n>\mathbf{n}$,
\[
B_i\subset \left[|\partial  X_{n,i}|\geq \left(1+\frac{\epsilon N}{4}\right)\mathbb{E}\left(\left|\partial R(m_i)\right|\right)\right],
\]
consequently for all $1\leq i\leq N,$
\begin{equation*}
\mathbb{P}\left(B_i\right)\leq \mathbb{P}\left(|\partial  R_{m_i}|\geq \left(1+\frac{\epsilon N}{4}\right)\mathbb{E}\left(\left|\partial R_{m_i}\right|\right)\right)\leq \frac{16C\left(\epsilon_0/4,\delta\right)}{\epsilon^2N^2\log^\delta(m_i)}.
\end{equation*}
Now as $N\sim \log^{\delta/3}(n)$ as $n\to\infty$, by enlarging $\mathbf{n}$ if needed, we can assume that for all $n>\mathbf{n}$ uniformly on $1\leq i\leq N$ and $\epsilon>\epsilon_0$,
\begin{equation}\label{eq: 3rd Flatto}
\mathbb{P}\left(B_i\right)\leq \frac{32C(\epsilon_0/4,\delta)}{\epsilon^2 N \log^{4\delta/3}(n)}.
\end{equation}
Since for $1\leq i< j\leq N$, the events $A_i$ and $A_j$ are independent it follows from \eqref{eq: 0th Flatto}, \eqref{eq: 2nd Flatto} and \eqref{eq: 3rd Flatto} that for all $\epsilon>\epsilon_0$ and $n>\mathbf{n}$,
\begin{align*}
\mathbb{P}\left(|\partial R(n)|>(1+\epsilon) \mathbb{E}\left(|\partial  R(n)|\right) \right) &\leq \sum_{1\leq i<j\leq N}\mathbb{P}(A_i)\mathbb{P}(A_j)+\sum_{k=1}^N \mathbb{P}(B_i)\\
&\leq \left(\frac{32C(\epsilon_0/4, \delta)}{\epsilon_0}\right)^2\frac{N^2}{\epsilon^2 \log^{2\delta}(n)}+\frac{32C(\epsilon_0/4,\delta)}{\epsilon^2  \log^{4\delta/3}(n)}\\
&\sim \left(\left(\frac{32C(\epsilon_0/4, \delta)}{\epsilon_0}\right)^2+32C(\epsilon_0/4,\delta)\right)\frac{1}{\epsilon^2  \log^{4\delta/3}(n)}
\end{align*}
as $n\to \infty$. It follows that  there exists $C(\epsilon_0,4\delta/3)$ such that for all $\epsilon\geq \epsilon_0$ and  $n\geq 2$
\[
\mathbb{P}\left(|\partial R(n)|>(1+\epsilon) \mathbb{E}\left(|\partial  R(n)|\right) \right)\leq \frac{C(\epsilon_0,4\delta/3)}{\epsilon^2\log^{4\delta/3}(n)}.
\]
As $\epsilon_0$ is arbitrary this concludes the proof of Theorem \ref{thm: Flatto}.

\section{Green functions tend to \texorpdfstring{$0$}{0} at infinity}
\label{appendix:Green}

In this appendix, we study the decay at infinity of the Green function of
transient random walks on groups. The setting is the same as in
Paragraph~\ref{subsec:transient_RW}: we consider a probability measure $p$ on
a group $\mathsf{G}$, and let $S_n = \xi_0\cdots \xi_{n-1}$ be the
corresponding random walk, where $\xi_i$ are i.i.d.~random variables
distributed like $p$. When this random walk is transient, the Green function
$G(g)$ of the walk is defined as in~\eqref{eq:define_Green}, by $G(g) =
\sum_{k=0}^\infty p_n(\id_{\mathsf{G}},g)$, where $p_n(\id_{\mathsf{G}}, g) =
\mathbb{P}\left(S_n=g\right)$. More generally, let $G(g, h) = G(g^{-1}h)$.
This is the average time that the walk starting from $g$ spends at $h$.

We say that $p$ is \emph{admissible} if the random walk can reach any point
in the group, i.e., the semigroup generated by the support of $p$ is the
whole group. This is a natural non-degeneracy assumption.

The main result of this appendix is the following theorem.

\begin{theorem}
\label{thm:Green_vanishes}
Let $p$ be an admissible probability measure on a finitely generated group
$\mathsf{G}$. Assume that $p$ defines a transient random walk, and that
$\mathsf{G}$ is not virtually cyclic. Then the Green function $G(g)$ tends to
$0$ when $g$ tends to infinity.
\end{theorem}
The condition that $\mathsf{G}$ is not virtually cyclic is necessary for the
theorem: in $\mathbb{Z}$, the Green function of a non-centered random walk
with finite first moment does not tend to $0$ at infinity, by the renewal
theorem (and this statement can be extended to virtually cyclic groups).

This theorem is rather weak, in the sense that it is not quantitative. Much
stronger results are available on specific classes of groups,
see~\cite{Woess}. For instance, when $\mathsf{G}$ is non-amenable, then the
probabilities $p_n(\id_{\mathsf{G}},x)$ tend to zero exponentially fast,
uniformly in $x$, from which the result follows readily. On amenable groups,
for symmetric walks with finite support or more generally a second moment,
one can sometimes use isoperimetric techniques to obtain much stronger
results, specifying the speed of decay of
$p_n(\id_{\mathsf{G}},\id_{\mathsf{G}})$ and of the Green function at
infinity. However, in general, there is no hope to get a quantitative version
of Theorem~\ref{thm:Green_vanishes} as we have made no moment assumption
(think of the case where $p$ is chosen so that $\xi_1$ is at distance
$2^{2^n}$ of $\id_{\mathsf{G}}$ with probability $1/n^2$ -- then the Green
function decays at most like $1/\log\log d(\id_{\mathsf{G}},x)$). Moreover,
there are surprisingly few tools that apply in all classes of groups,
regardless of their geometry.

Most groups only carry transient random walks. Indeed, the only groups on
which there are recurrent random walks are the groups which are virtually
cyclic or virtually $\mathbb{Z}^2$, see~\cite[Theorem 3.24]{Woess}. In the
proof of the theorem, we will have to separate the case where $\mathsf{G}$ is
virtually $\mathbb{Z}^2$. Let us start with this case.

\begin{lemma}
\label{lem:Z2} Assume that $\mathsf{G}$ is virtually $\mathbb{Z}^2$, and that
the admissible probability measure $p$ on $\mathsf{G}$ defines a transient
random walk. Then its Green function tends to $0$ at infinity.
\end{lemma}
\begin{proof}
Assume first that $\mathsf{G} = \mathbb{Z}^2$. Then the convergence to $0$ at
infinity of the Green function is \cite[24.P.5]{Spitzer}, as $p$ is
admissible and therefore aperiodic.

Assume now that $\mathsf{G}$ has a finite index subgroup $H$ which is
isomorphic to $\mathbb{Z}^2$. Replacing $H$ with the intersection of its
(finitely many) conjugates, one can even assume that $H$ is normal in
$\mathsf{G}$. The measure $p$ induces a measure on the group $\mathsf{G}/H$,
which defines a recurrent walk as $\mathsf{G}/H$ is finite. In particular,
almost every trajectory of the random walk returns to $H$. The distribution
of this first return is an admissible probability measure $p_H$ on $H$, to
which one can apply the previous result: its Green function tends to $0$ at
infinity. Moreover, the Green functions of $p$ and $p_H$ coincide on $H$ as
the trajectories of the random walk associated to $p_H$ can be obtained from
the trajectories of the random walk for $p$ by restricting to the times where
the walk is in $H$. It follows that $G(g)$ tends to $0$ when $g$ tends to
infinity along $H$.

By Harnack's inequality \cite[25.1]{Woess}, there exists a constant $C$ such
that, for all $g_1,g_2\in \mathsf{G}$, one has $G(g_1) \leq C^{d(g_1, g_2)}
G(g_2)$. As $H$ has finite index in $\mathsf{G}$, every point of $\mathsf{G}$
is within uniformly bounded distance of $H$. Therefore, the convergence to
$0$ of the Green function along $H$ extends to the whole group.
\end{proof}

To prove Theorem~\ref{thm:Green_vanishes}, we can therefore restrict to the
case where the simple random walk is transient. We will use its symmetry to
obtain the decay at infinity of its Green function. Then, we will compare a
general random walk to the simple random walk.

\begin{lemma}
\label{lem:symmetric_0}
Assume that a symmetric random walk on a finitely generated group $\mathsf{G}$ is transient.
Then its Green function tends to $0$ at infinity.
\end{lemma}
\begin{proof}
Let $p_n(g, h)$ denote the probability that the walk starting at $g$ is at
position $h$ at time $n$. Then it is a standard fact that
$p_{2n}(\id_{\mathsf{G}}, g)$ is maximal for $g=\id_{\mathsf{G}}$. This is
proved using Cauchy-Schwarz inequality as follows:
\begin{align*}
p_{2n}(\id_{\mathsf{G}},g) & = \sum_h p_n(\id_{\mathsf{G}}, h) p_n(h, g)
  \leq \pare*{ \sum_h p_n(\id_{\mathsf{G}},h)^2}^{1/2} \pare*{ \sum_h p_n(h,g)^2}^{1/2}
  \\&
  = \pare*{ \sum_h p_n(\id_{\mathsf{G}},h) p_n(h,\id_{\mathsf{G}})}^{1/2} \pare*{ \sum_h p_n(g,h) p_n(h,g)}^{1/2}
  \\&
  = p_{2n}(\id_{\mathsf{G}},\id_{\mathsf{G}})^{1/2} p_{2n}(g,g)^{1/2}
  = p_{2n} (\id_{\mathsf{G}},\id_{\mathsf{G}}).
\end{align*}
Conditioning on the position of the walk at time $1$, one also gets
$p_{2n+1}(\id_{\mathsf{G}},g) \leq p_{2n}(\id_{\mathsf{G}},
\id_{\mathsf{G}})$. Therefore, for any $g$ and any $N$,
\begin{equation*}
  \sum_{n=2N}^\infty p_n(\id_{\mathsf{G}}, g) \leq
  2 \sum_{n=2N}^\infty p_n(\id_{\mathsf{G}}, \id_{\mathsf{G}}).
\end{equation*}
The right hand side is the tail of the converging series $\sum_n
p_n(\id_{\mathsf{G}}, \id_{\mathsf{G}}) = G(\id_{\mathsf{G}})$. If $N$ is
large enough, it is bounded by an arbitrarily small constant $\epsilon$. For
each $n<2N$, the measure $p_n$ is a probability measure on $\mathsf{G}$.
Hence, there are only finitely many points $g$ for which
$p_n(\id_{\mathsf{G}},g) > \epsilon/(2N)$. Summing over $n < 2N$, it follows
that for all but finitely many points one has $\sum_{n<2N}
p_n(\id_{\mathsf{G}},g) \leq \epsilon$, and therefore $G(g) = \sum_{n}
p_n(\id_{\mathsf{G}},g) \leq 2\epsilon$.
\end{proof}

The next lemma is the main step of the proof of Theorem~\ref{thm:Green_vanishes}.
\begin{lemma}
\label{lem:bounded_below_on_S} Consider a finitely generated group
$\mathsf{G}$, with a finite generating set $S$. Assume that the simple random
walk on $\mathsf{G}$ is transient. Let $c>0$. Consider a probability measure
$p$ on $\mathsf{G}$ with $p(s) \geq c$ for all $s\in S$. Then its Green
function tends to $0$ at infinity.
\end{lemma}
\begin{proof}
The proof will rely on a classical comparison lemma, making it possible to relate general random
walks to symmetric ones. Denote by $G$ the Green function associated to $p$, and by $G_S$ the Green function
associated to the simple random walk. Under the assumptions of the lemma, \cite[Proposition, Page 251]{Varopoulos}
ensures that for any nonnegative square-integrable function $f$ on $\mathsf{G}$,
\begin{equation}
\label{eq:varopoulos}
  \sum_{g,h \in \mathsf{G}} f(g) G(g,h) f(h) \leq c^{-1} \sum_{g,h \in \mathsf{G}} f(g) G_S(g,h) f(h).
\end{equation}

Let $\Omega = \mathsf{G}^{\mathbb{N}}$ be the space of all possible
trajectories, with the (Markov) probability $\mathbb{P}$ corresponding to the
distribution of the random walk given by $p$ starting from
$\id_{\mathsf{G}}$. A cylinder set is a set of the form
\[
  [g_0,\dotsc, g_n] = \{\omega \in \Omega \mid \omega_0 = g_0,\cdots,\omega_n = g_n\} \subseteq \Omega.
\]

Assume by contradiction that $G(g)$ does not tend to $0$ at infinity. Then
one can find $\epsilon>0$ and an infinite set $I \subseteq G$ on which $G(g)
> \epsilon G(\id_{\mathsf{G}})$. Let $M>0$ be large enough (how large will be specified at
the end of the argument). We define a sequence $h_n$ of elements of $I$ as
follows.
\begin{itemize}
\item First, take $h_0 = \id_{\mathsf{G}}$. Let also $T_0 =
    [\id_{\mathsf{G}}] \subseteq \Omega$ be the set of all trajectories
    starting from $\id_{\mathsf{G}}$, and $R_0 = \{\id_{\mathsf{G}}\}$.
\item Then, take $h_1 \in I$ at distance at least $M$ of $h_0$. As
    $G(h_1)/G(\id_{\mathsf{G}})>\epsilon$ is the probability to reach $h_1$
    starting from $\id_{\mathsf{G}}$, one can find (by throwing away very
    long trajectories or very unlikely trajectories in finite time) a
    finite number of cylinder sets starting at $\id_{\mathsf{G}}$ and
    ending at $h_1$ with total probability $>\epsilon$. Denote this set of
    trajectories by $T_1$, with $\mathbb{P}(T_1) > \epsilon$. Let $R_1$ be
    the set of points that trajectories in $T_1$ reach before $h_1$. As
    $T_1$ is a finite union of cylinders, $R_1$ is finite.
\item Then, take $h_2 \in I$ at distance at least $M$ of $h_0$ and $h_1$.
    We can also require that it does not belong to $R_0 \cup R_1$, as this
    set is finite while $I$ is infinite. As above, we can then define a set
    $T_2$ which is a finite union of cylinders ending at $h_2$ with
    $\mathbb{P}(T_2)>\epsilon$, and $R_2$ the finite set of points reached
    by these trajectories before $h_2$.
\item The construction goes on inductively to define $h_n$.
\end{itemize}
Let us denote by $F(g, h) = G(g, h)/G(\id_{\mathsf{G}})$ the probability to
reach $h$ starting from $g$. The point of the previous construction is that
for $i\leq j$ we have the inequality
\begin{equation}
\label{eq:traj_below}
  \mathbb{P}(T_i \cap T_j) \leq F(h_i, h_j).
\end{equation}
Indeed, this is clear for $i=j$. For $i<j$, note that trajectories in $T_i
\cap T_j$ reach $h_i$ and then $h_j$, in this order as $h_j \notin R_i$.
Therefore,
\begin{equation*}
  \mathbb{P}(T_i \cap T_j) \leq \mathbb{P}(\exists n< m, S_n = h_i \text{ and } S_m = h_j)
  =F(\id_{\mathsf{G}}, h_i) \cdot F(h_i, h_j) \leq F(h_i, h_j),
\end{equation*}
where the central equality follows from the Markov property. This proves~\eqref{eq:traj_below}.

Let us now take $N$ large, and apply the inequality~\eqref{eq:varopoulos} to
the characteristic function of $\{h_0,\cdots, h_{N-1}\}$. We obtain
\begin{equation}
\label{eq:varopoulos2}
  \sum_{i,j < N} G(h_i, h_j) \leq c^{-1} \sum_{i, j < N} G_S(h_i, h_j).
\end{equation}
We will bound the left hand side from below and the right hand side from
above to get a contradiction. Thanks to~\eqref{eq:traj_below}, we have
\begin{align*}
  \sum_{i,j < N} G(h_i, h_j) & \geq \sum_{i \leq j < N} G(h_i, h_j)
  = G(\id_{\mathsf{G}})  \sum_{i \leq j < N} F(h_i, h_j)
  \geq G(\id_{\mathsf{G}}) \sum_{i\leq j < N} \mathbb{P}(T_i \cap T_j)
  \\&
  \geq \frac{G(\id_{\mathsf{G}})}{2} \sum_{i, j < N} \mathbb{P}(T_i \cap T_j)
  = \frac{G(\id_{\mathsf{G}})}{2} \int \pare*{\sum_{i<N} 1_{T_i}}^2 \rd\mathbb{P}
  \geq \frac{G(\id_{\mathsf{G}})}{2} \pare*{ \int \sum_{i<N} 1_{T_i} \rd\mathbb{P}}^2
  \\&
  = \frac{G(\id_{\mathsf{G}})}{2} \pare*{ \sum_{i<N} \mathbb{P}(T_i)}^2
  \geq \frac{G(\id_{\mathsf{G}})}{2} \epsilon^2 N^2.
\end{align*}

The Green function $G_S$ tends to $0$ at infinity, by Lemma~\ref{lem:symmetric_0}. As the distance
between $h_i$ and $h_j$ is at least $M$ for $i\neq j$ by construction, it follows that
$G_S(h_i, h_j) \leq \eta(M)$ where $\eta$ tends to $0$ with $M$. We obtain
\begin{align*}
  \sum_{i, j < N} G_S(h_i, h_j)
  \leq \sum_{i=j} G_S(h_i, h_j) + \sum_{i \neq j} G_S(h_i, h_j)
  \leq N G_S(\id_{\mathsf{G}}) + N^2 \eta(M).
\end{align*}
Combining these two estimates with~\eqref{eq:varopoulos2} yields
\begin{equation*}
  \frac{G(\id_{\mathsf{G}})}{2} \epsilon^2 N^2 \leq c^{-1} N G_S(\id_{\mathsf{G}}) + c^{-1} N^2 \eta(M).
\end{equation*}
We obtain a contradiction by taking $M$ large enough so that $c^{-1} \eta(M)
< G(\id_{\mathsf{G}}) \epsilon^2/2$, and then letting $N$ tend to infinity.
\end{proof}

\begin{proof}[Proof of Theorem~\ref{thm:Green_vanishes}]
The result follows from Lemma~\ref{lem:Z2} if $\mathsf{G}$ is virtually
$\mathbb{Z}^2$. Hence, we can assume that this is not the case, and therefore
that the simple random walk on $\mathsf{G}$ is transient.

The Green functions for the probability measures $p$ and
$(p+\delta_{\id_{\mathsf{G}}})/2$ are related by the identity
$G_{(p+\delta_{\id_{\mathsf{G}}})/2}(g) = 2 G_p(g)$,
by~\cite[Lemma~9.2]{Woess}. Without loss of generality, we can therefore
replace $p$ with $(p+\delta_{\id_{\mathsf{G}}})/2$ and assume
$p(\id_{\mathsf{G}})>0$. As $p$ is admissible, it follows that there exists
$N$ with $p_N(\id_{\mathsf{G}},s)>0$ for all $s$ in the generating set $S$.
By Lemma~\ref{lem:bounded_below_on_S}, the Green function associated to
$p_N$, denoted by $G_N$, tends to $0$ at infinity.

To compute $G(g)$, split the arrival times at $g$ according to their values
modulo $N$. For times of the form $i + kN$, such arrivals can be realized by
following $p$ for $i$ steps, and then $p_N$ for $k$ steps. It follows that
\begin{equation*}
  G(g) = \sum_{h \in \mathsf{G}} \sum_{i<N} p_i(\id_{\mathsf{G}},h) G_N(h^{-1}g).
\end{equation*}
Let $\epsilon>0$. Take a finite set $F \subset \mathsf{G}$ such that $\sum_{h
\notin F}\sum_{i<N} p_i(\id_{\mathsf{G}},h) < \epsilon$. Then
\begin{equation*}
  G(g) \leq \sum_{h \in F} \sum_{i<N} p_i(\id_{\mathsf{G}},h) G_N(h^{-1}g) + \epsilon \|G_N\|_{L^\infty}.
\end{equation*}
When $g$ tends to infinity, the first term tends to $0$ as this is a finite
sum and $G_N$ tends to $0$ at infinity. For large enough $g$, wet get $G(g)
\leq 2\epsilon  \|G_N\|_{L^\infty}$.
\end{proof}

\bibliographystyle{halpha}
\bibliography{myBIB}

\end{document}